\documentclass[12pt]{article}
\usepackage[T1]{fontenc}
\usepackage[a4paper, margin=2.7cm]{geometry}
\usepackage{amsmath, amsthm, amsfonts, amssymb}
\usepackage{dsfont}
\usepackage{stackrel,stmaryrd}
\usepackage{hyperref}
\usepackage{mathrsfs}

\newtheorem{theorem}{Theorem}[section]
\newtheorem{lemma}{Lemma}[section]
\newtheorem{corollary}{Corollary}[section]

\theoremstyle{remark}
\newtheorem{remark}{Remark}[section]
\theoremstyle{definition}
\newtheorem{definition}{Definition}[section]

\title{Multivalued backward stochastic differential equations with jumps and moving boundary}
\author{Badr ELMANSOURI\footnote{Faculty of Sciences Agadir, Ibn Zohr University,
		Laboratory of Analysis and Applied Mathematics (LAMA),
		Hay Dakhla, BP8106, Agadir, Morocco.} \footnote{Email: \url{badr.elmansouri@edu.uiz.ac.ma}}  \and Anas OUKNINE\footnote{Laboratoire de Mathématiques Raphaël Salem, University of Rouen, UMR CNRS 6085, 
		Avenue de l'Université, 76801 Saint-Étienne-du-Rouvray, France.} \footnote{Email: \url{ anas.ouknine@univ-lemans.fr}}  \and Youssef OUKNINE\footnote{Department of Mathematics, Faculty of Sciences Semlalia, Cadi Ayyad University, B.P. 2390, Marrakesh, Morocco.} \footnote{Africa Business School, Mohammed VI Polytechnic University, Lot 660, Hay Moulay Rachid,
		Ben Guerir 43150, Morocco.}     \footnote{Emails: \url{ouknine@uca.ac.ma} \& \url{youssef.ouknine@um6p.ma}} }

\begin{document}
\maketitle
\begin{abstract}
	We prove existence and uniqueness for a one-dimensional multivalued backward stochastic differential equation with jumps. The equation involves a time-indexed family of maximal monotone operators $k_t(\cdot)$ associated with increasing functions $k(t,\cdot)$ taking values in $\mathbb{R}_-$ and having domains that are intervals with time-dependent boundaries. Existence is obtained by a penalization method under a Lipschitz condition on the driver in $(y,z)$, a monotonicity condition in the jump parameter $\psi$, square-integrability of the terminal condition and the driver, and local-in-time integrability conditions on $k(\cdot,y)$. We also address the extension to the case where the operators $k_t(\cdot)$ act on unbounded intervals.
\end{abstract}


\section{Introduction}
The notion of backward stochastic differential equations (BSDEs) was introduced in its linear form by Bismut \cite{bismut1973conjugate} in 1973 as the adjoint equations associated with stochastic Pontryagin maximum principles in control theory. The general nonlinear case was first studied by Pardoux and Peng \cite{pardoux1990adapted} in 1990. For a given terminal time $T\in(0,+\infty)$, a terminal condition $\xi$, and a driver $f(\omega,t,y,z)$, they proved the existence and uniqueness of an adapted pair of processes $(Y_t,Z_t)_{t\in[0,T]}$ satisfying
\begin{equation}\label{01}
	Y_t = \xi + \int_t^T f(s,Y_s,Z_s)ds - \int_t^T Z_s dW_s,
	\quad t\in[0,T],
\end{equation}
where $W=(W_t)_{t\in[0,T]}$ is a standard Brownian motion. Their proof relies on the martingale representation property in the filtration generated by $W$, under a Lipschitz condition on the driver $f$ and square‐integrability assumptions on $\xi$ and the process $(f(t,0,0))_{t\in[0,T]}$. Since then, interest in BSDEs has grown substantially due to their wide‐ranging applications, notably in mathematical finance \cite{el1997backward}, in providing probabilistic representations of solutions for quasilinear partial differential equations \cite{pardoux1998backward,Pardoux1999,PardouxPeng1992}, and in stochastic control and differential games \cite{HamadeneLepeltier1995,HamadeneLepeltierPeng1997}.

The extension to the case with jumps, where the filtration $\mathbb{F}=(\mathcal{F}_t)_{t\in[0,T]}$ is generated by the Brownian motion $W$ and an independent Poisson random measure $N(dt,de)$ on $[0,T] \times \mathcal{U}$ where $\mathcal{U} := \mathbb{R}^k \setminus \{0\}$ ($k \geq 1$), has been studied by several authors \cite{Delong2014,rong1997solutions,royer2006backward,TangLi1994}. In this setting, BSDE \eqref{01} takes the form
\begin{equation}\label{eq02}
	Y_t = \xi + \int_t^T f(s,Y_s,Z_s,\psi_s)ds 
	-\int_t^T Z_sdW_s 
	-\int_t^T\int_{\mathcal{U}} \psi_s(e)\tilde{N}(ds,de),
	\quad t\in[0,T].
\end{equation}
Such equations arise naturally in various applications: in finance, for pricing contingent claims in jump–diffusion models extending the classical Black–Scholes framework \cite{Situ2005}; in dynamic risk measurement and utility‐maximization problems \cite{Becherer2006,Morlais2010,royer2006backward}; and in the probabilistic representation of solutions to partial differential equations \cite{BarlesBuckdahnPardoux1997,Pardoux1997}.

The aim of this paper is to establish existence and uniqueness results for multivalued BSDEs (MBSDEs) analogous to those obtained for BSDEs of the form \eqref{eq02}. Here, we consider a time‐indexed family of multivalued maximal monotone operators on $\mathbb{R}$, denoted by $\{k_t(\cdot)\colon t\in[0,T]\}$, associated with a family of right‐continuous increasing functions $\{k(t,\cdot)\colon t\in[0,T]\}$. Recall that maximal monotone operators on $\mathbb{R}$ are precisely the subdifferentials of proper, lower semicontinuous convex functions (see Brézis \cite{Brezis1973}), or equivalently, multivalued operators generated by increasing functions on intervals with nonempty interior. In \cite{NziOuknine1997}, N’zi and Ouknine proved an existence and uniqueness result for an MBSDE driven by a single maximal monotone operator, with square‐integrable data and a Lipschitz driver, in the continuous Brownian setting. Their approach combines Yosida approximations of maximal operators with a backward Skorokhod problem. This result was extended by the same authors to continuous drivers with linear growth in \cite{NziOuknine1997b}, and by N’zi to the case of locally Lipschitz drivers and bounded terminal conditions \cite{Nzi1997,Nzi1997c}.

Following the progressive‐filtration framework of Marois \cite{marois1990equations}, we address the discontinuous case with Poisson jumps and a time‐indexed family of multivalued operators $k_t(\cdot)$, associated with increasing functions $k(t,\cdot)$ defined on intervals $\mathcal{D}_t$ with interior $[a_t,+\infty[$. Here, the lower boundary $a_t$ evolves continuously in time ($t\mapsto a_t$ is continuous on $[0,T]$), and the state process $Y$ of the MBSDE takes values in the moving domain $[a_t,+\infty[\cap\mathbb{R}$ at each time $t$. For the BSDE data, we assume a square‐integrable terminal condition $\xi\in[a_T,+\infty[$ and a driver $f$ that is Lipschitz in $(y,z)$, satisfies a monotonicity condition in $\psi$, and for which $(f(t,0,0,0))_{t\le T}$ is square‐integrable. The monotonicity in $\psi$ ensures a comparison theorem for our MBSDEs. Indeed, comparison can fail in general for discontinuous BSDEs of the form \eqref{eq02} (see the counterexample in Barles et al. \cite{BarlesBuckdahnPardoux1997}). Under an additional monotonicity assumption on $f$ with respect to $\psi$ (as in \cite{royer2006backward}), we establish a comparison principle. We first focus on the case where the functions $k(t,\cdot)$ take negative values, so that the graph of each operator lies in $\mathbb{R}\times\mathbb{R}_-$. We then discuss the general case and prove an existence result via localization and concatenation procedures. For related work, see \cite{essaky2004homogenization,Ouknine1998}.

The paper is organized as follows. In Section \ref{sec1}, we introduce our stochastic basis, formulate the problem, and state the main assumptions on the data $(\xi,f,k)$ for the MBSDE. In Section \ref{sec1-1}, we prove a uniqueness theorem and establish a comparison principle. Section \ref{sec2} is devoted to constructing a solution $(Y,Z,\psi,K)$ when the graphs of the operators $k_t(\cdot)$ lie in $\mathbb{R}\times\mathbb{R}_-$. Using a penalization method, we build a sequence of processes $\{Y^n,Z^n,\psi^n,K^n\}_{n\ge1}$ with $Y^n$ increasing and converging to $Y$. Under two local‐in‐time integrability conditions on $k(\cdot,x)$, we show that $(Z^n,\psi^n)$ converges in the appropriate Banach space to $(Z,\psi)$, and that $K$ is obtained as the limit of a subsequence of $K^n := -\int_0^T k_n\bigl(s,Y^n_s\bigr)ds$ where $k_n(t,\cdot)$ are Lipschitz approximations of $k(t,\cdot)$. Finally, in Section \ref{sec4}, we extend the existence and uniqueness results to the general case where $k(t,\cdot)$ takes real values. This extension is achieved by truncating the graph of $k_t(\cdot)$ and concatenating the corresponding local solutions, under an additional assumption on the  functions $k(t,\cdot)$.

\section{Preliminaries}
\label{sec1}
Let $(\Omega,\mathcal{F},\mathbb{P})$ be a complete probability space equipped with a $\mathbb{P}$-completed, right-continuous filtration $\mathbb{F}:=(\mathcal{F}_t)_{t \geq 0}$. Let
\begin{itemize}
	\item $(W_t)_{t \geq 0}$ be a one-dimensional standard Brownian motion.
	
	\item $\tilde{N}(dt,de)$ be an independent martingale measure associated with a standard Poisson random measure $N$ on $\mathbb{R}^{+} \times \mathcal{U}$, with Lévy (or characteristic) measure $\pi(de)$ \cite{IkedaWatanabe1989,Situ2005}, where $\mathcal{U} := \mathbb{R}^k \setminus \{0\}$ ($k \geq 1$) is equipped with its Borel $\sigma$-algebra $\mathscr{B}(\mathcal{U})$. Specifically, for any Borel measurable set $\Lambda \in \mathscr{B}(\mathcal{U})$ such that $\pi(\Lambda) < +\infty$, we define
	$$
	\tilde{N}((0,t] \times \Lambda) := N((0,t] \times \Lambda) - t \pi(\Lambda),
	$$
	where $\pi$ is a $\sigma$-finite measure on $(\mathcal{U}, \mathscr{B}(\mathcal{U}))$ satisfying
	$$
	\int_{\mathcal{U}} \left(1 \wedge |e|^2\right) \pi(de) < +\infty.
	$$
\end{itemize}

We assume that 
$$
\mathcal{F}_t = \sigma\left\{\int_{0}^{s} \int_A N(dr,du);\, s \in [0,t],\, A \in \mathscr{B}({\mathcal{U}})\right\} \bigvee \sigma\left\{W_s;\, s \in [0,t]\right\} \bigvee \mathcal{N},
$$
where $\mathcal{N}$ denotes the collection of $\mathbb{P}$-null sets in $\mathcal{F}$, and $\sigma_1 \vee \sigma_2$ denotes the $\sigma$-field generated by $\sigma_1 \cup \sigma_2$. We denote by $\mathcal{P}$ the predictable $\sigma$-field on $\Omega \times \mathbb{R}_+$.

By $\mathbb{L}^2_\pi$ we denote the space of $\mathscr{B}({\mathcal{U}})$-measurable functions $\varphi : \mathcal{U} \rightarrow \mathbb{R}$ such that
$$
\|\varphi\|^2_\pi:= \int_\mathcal{U}|\varphi(u)|^2 \pi(du)<+\infty.
$$

The following objects are given:
\begin{itemize}
	\item A terminal time $T \in (0,+\infty)$;
	
	\item A family of increasing, right-continuous mappings $\{k(t, \cdot)\}_{t \in [0,T]}$, taking negative values and defined on domains $\mathcal{D}_t$ with interior $\mathcal{D}^\circ_t := ]a_t; +\infty[$. For each $t \in [0,T]$, the boundary point $a_t$ is said to belong to $\mathcal{D}_t$ when
	$$
	a_t > -\infty \quad \text{and} \quad \lim\limits_{x \uparrow a_t} k(t,x) > -\infty.
	$$
	At each interior point $x \in \mathcal{D}^\circ_t$, the function $k(t,\cdot)$ admits a left-limit, denoted $k_-(t, x)$.
	
	Recall that maximal monotone operators on $\mathbb{R}$ are simply multivalued operators associated with increasing functions defined on intervals with non-empty interior (see \cite[Example 2.3.1, page 24]{Brezis1973}). For each $t \in [0,T]$, the maximal monotone operator $k_t(\cdot)$ associated with $k(t,\cdot)$, defined on $\mathcal{D}_t$, is given as follows:
	\begin{itemize}
		\item[$\ast$] For every $x \in \mathcal{D}^\circ_t = ]a_t; +\infty[$, we set
		$$
		k_t(x) = \left[k_-(t,x); k(t,x)\right].
		$$
		
		\item[$\ast$] If $a_t \in \mathcal{D}_t$, we define
		$$
		k_t(a_t) = \left]-\infty ; k(t,a_t)\right].
		$$
	\end{itemize}
	The functions $k(\cdot, x)$ are assumed to be measurable on $\mathbb{R}_+$. This assumption will be maintained throughout the rest of the paper.
	
	\item A terminal value $\xi$ and a function (driver or coefficient) $f : \Omega \times [0,T] \times \mathbb{R} \times \mathbb{R} \times \mathbb{L}^2_\pi \rightarrow \mathbb{R}$ satisfying the following assumptions (A):
	\begin{itemize}
		\item[(A.1)] $\xi$ is an $\mathcal{F}_T$-measurable, square-integrable random variable taking values in $[a_T; +\infty[ \; \cap \; \mathbb{R}$.
		
		\item[(A.2)] 
		\begin{itemize}
			\item[(i)] For any $(y,z,\psi) \in \mathbb{R} \times \mathbb{R} \times \mathbb{L}^2_\pi$, the mapping $\Omega \times [0,T] \ni (\omega,t) \mapsto f(\omega,t, y,z,\psi)$ is $\mathbb{F}$-progressively measurable;
			
			\item[(ii)] $\mathbb{E}\left[\int_{0}^{T} |f(t,0,0,0)|^2 dt \right] < +\infty$;
			
			\item[(iii)] There exists a constant $\widetilde{C} > 0$ such that, for any $(\omega, t) \in \Omega \times [0,T]$ and all $(y,y',z,z',\psi) \in \mathbb{R}^2 \times \mathbb{R}^2 \times \mathbb{L}^2_\pi$, we have
			$$
			\left|f(t,y, z, \psi) - f(t,y', z', \psi)\right| \leq \widetilde{C} \left(|y - y'| + |z - z'| \right);
			$$
			
			\item[(iv)] For each $(y, z, \psi, \phi) \in \mathbb{R} \times \mathbb{R} \times \mathbb{L}^2_\pi \times \mathbb{L}^2_\pi$, there exists a predictable process $\kappa = \kappa^{y, z, \psi, \phi} : \Omega \times [0, T] \times \mathcal{U} \rightarrow \mathbb{R}$ such that
			$$
			f(t, y, z, \psi) - f(t, y, z, \phi) \leq \int_{\mathcal{U}} (\psi(e) - \phi(e)) \kappa_t^{y, z, \psi, \phi}(e)\, \pi(de),
			$$
			with the following conditions holding $\mathbb{P} \otimes dt \otimes \pi$-a.e. for any $(y, z, \psi, \phi)$:
			\begin{itemize}
				\item $\kappa_t^{y, z, \psi, \phi}(e) \geq -1$,
				
				\item $\left|\kappa_t^{y, z, \psi, \phi}(e)\right| \leq \vartheta(e)$, where $\vartheta \in \mathbb{L}_\pi^2$.
			\end{itemize}
		\end{itemize}
	\end{itemize}
\end{itemize}

\begin{remark}
	If condition (A.2)-(iv) is true and by changing the role of $\psi$ and $\phi$ in $\kappa$, then we have
	$$
	f(t, y, z, \psi)-f(t, y, z, \phi) \geq \int_{\mathcal{U}}(\psi(e)-\phi(e)) \kappa_t^{y, z, \phi, \psi}(e) \pi({d} e)
	$$
Thus
$$
\left|f(t, y, z, \psi)-f(t, y, z, \phi)\right| \leq \| \vartheta \|_\pi \|\psi-\phi\|_\pi.
$$
\end{remark}
Following the above remark, we can deduce from (A.2)-(iii) that the driver $f$ is uniformly Lipschitz with respect to the variables $(y, z, \psi)$ with constant $C := \widetilde{C} \vee \| \vartheta \|_\pi$. To simplify notation, we will use the constant $C>0$ to denote the Lipschitz constant of $f$ throughout the remainder of the paper.

The problem consists in finding a quadruplet $(Y, Z, U, K):=(Y_t, Z_t, U_t, K_t)_{t \leq T}$ of $\mathbb{F}$-progressively measurable processes such that:

\begin{align}
	Y_t = \xi + \int_t^T f(s, Y_s, Z_s, \psi_s)\, ds - \int_t^T Z_s\, dW_s - \int_t^T\int_{\mathcal{U}} \psi_s(e)\, \tilde{N}(ds,de) + (K_T - K_t)&, \label{eq1} \\
	Y \text{ is an RCLL process taking its values at each time } t \in [0,T] \text{ in } \mathbb{R} \cap [a_t, +\infty[&, \label{eq2} \\
	K \text{ is continuous, increasing with } K_0 = 0&, \label{eq3} \\
	\begin{aligned}
		&\text{For every pair of optional processes } (\alpha_t , \beta_t) \text{ with values in } \operatorname{Gr}(k_t) \text{ for } t \in [0,T], \\
		&\text{the measure } (Y_t - \alpha_t)(dK_t + \beta_t dt) \text{ is a.s. negative on } [0,T].
	\end{aligned} \label{eq4}
\end{align}

Here, $\operatorname{Gr}(k_t)$ denotes the graph of the operation $k_t(\cdot)$ for $t \in [0,T]$ defined by 
$$
\operatorname{Gr}(k_t) := \left\{ (x, y) \in \mathbb{R} \times \mathbb{R} \;\middle|\;
\begin{array}{ll}
	y \in [k_-(t, x),\; k(t, x)] & \text{for } x \in \mathcal{D}_t^\circ = ]a_t, +\infty[, \\
	y \leq k(t, a_t) & \text{if } a_t \in \mathcal{D}_t \text{ and } x = a_t
\end{array}
\right\}.
$$

We denote by MBSDE$(\xi, f, k)$ the multivalued backward stochastic differential equation (MBSDE) associated with $(\xi, f, k)$, as defined by equations \eqref{eq1}--\eqref{eq4}.

\begin{definition}\label{def}
We call a quadruplet of $\mathbb{F}$-progressively measurable processes $(Y, Z, U, K)$ a solution to the MBSDE$(\xi, f, k)$ if it satisfies \eqref{eq1}--\eqref{eq4} and the following:
$$
\mathbb{E}\left[\sup_{0 \leq t \leq T} |Y_t|^2 + \int_{0}^{T} \left( |Z_s|^2 + \|\psi_s\|^2_\pi \right) ds + | K_T|^2 \right] < +\infty.
$$
\end{definition}

\begin{remark}
	The formulation of BSDE \eqref{eq1} together with condition \eqref{eq4} parallels the reflected BSDE studied in the Brownian setting by El Karoui et al. \cite{el1997reflected} and in our stochastic basis by Essaky \cite{essaky2008reflected} and by Hamadène and Ouknine \cite{SOUK,hamadeene2016reflected}. Here, the lower obstacle $L=(L_t)_{t\le T}$ is given by the continuous function $L_t(\omega)=a_t$. Since the jumps of $Y$ are totally inaccessible and arise only from the Poisson random measure $N$, the process $K$ remains continuous. Moreover, the classical Skorokhod condition
	$$
	\int_0^T (Y_s - a_s)\,dK_s = 0\quad \text{a.s.,}
	$$
	is replaced by condition \eqref{eq4}.
\end{remark}
\begin{remark}
	Throughout this paper, $\mathfrak{C}$ will denote a positive constant that may vary from one line to another.
	Additionally, the notation $\mathfrak{C}_\gamma$ will be employed to emphasize the dependence of the constant $\mathfrak{C}$ on a specific set of parameters $\gamma$.
\end{remark}

\section{Uniqueness result and comparison principal}
\label{sec1-1}
To prove the uniqueness and comparison theorems, we begin by stating several auxiliary results.

We begin with the following lemma:
\begin{lemma}\label{lemma1}
	Let the quadruples $(Y^1, Z^1, \psi^1, K^1)$ and $(Y^2, Z^2, \psi^2, K^2)$ be two solutions of the MBSDE$(\xi, f, k)$. Then, the measure $\left(Y^1_t - Y^2_t\right)\left(dK^1_t - dK^2_t\right)$ is a.s. negative on $[0,T]$.
\end{lemma}
\begin{proof}
	The proof of this lemma follows the same lines as that given by L{\'e}pingle and Marois \cite{lepingle2006equations} for multivalued stochastic differential equations in the Brownian setting, and later by Marois \cite{marois1990equations} for the discontinuous case. However, for the sake of completeness, we provide the full proof here.

	From \eqref{eq2}, we know that $Y^1_t$ and $Y^2_t$ take values in $[a_t; +\infty[$. The idea is to take, as a candidate for $\alpha_t$ of \eqref{eq4}, the process $(Y^1 + Y^2)/2$ when it is distinct from the boundary $a_t$ (which is the case if $Y^1$ and $Y^2$ are distinct), and to take $\beta_t=k(t,(Y^1_t + Y^2_t)/2)$. Otherwise, one may take, for instance, as a pair of optional processes belonging to $\operatorname{Gr}(k_t)$, the pair $(\widehat{\alpha}_t,\widehat{\beta}_t):=(a_t+\varepsilon, k(t,\widehat{\alpha}_t)$ where $\varepsilon$ is any strictly positive constant.\\
	Then, by considering the pair $(\alpha_t, \beta_t)_{t \in [0,T]}$ defined by:
	\begin{equation*}
		\left\lbrace 
		\begin{split}
			\alpha_t&:=\frac{Y^1_t+Y^2_t}{2}\mathds{1}_{\{Y^1_t \neq Y^2_t\}}+\widehat{\alpha}_t \mathds{1}_{\{Y^1_t = Y^2_t\}};\\
			\beta_t&:=k\left(t,\frac{Y^1_t+Y^2_t}{2}\right)\mathds{1}_{\{Y^1_t \neq Y^2_t\}}+\widehat{\beta}_t \mathds{1}_{\{Y^1_t = Y^2_t\}}.
		\end{split}
		\right. 
	\end{equation*}
From \eqref{eq4}, we derive the negativity of the following measures on $[0,T]$:
	 \begin{equation*}
	 	\left\lbrace 
	 	\begin{split}
	 		&\mathds{1}_{\{Y^1_t \neq Y^2_t\}}\left(Y^1_t-Y^2_t\right)\left(dK^1_t+k\left(t,\frac{Y^1_t+Y^2_t}{2}\right)dt \right);\\
	 		&\mathds{1}_{\{Y^1_t \neq Y^2_t\}}\left(Y^2_t-Y^1_t\right)\left(dK^2_t+k\left(t,\frac{Y^1_t+Y^2_t}{2}\right)dt \right).
	 	\end{split}
	 	\right. 
	 \end{equation*}
This, together with the following equality in the sense of measures:
\begin{equation*}
	\begin{split}
		&\left(Y^1_t - Y^2_t\right)\left(dK^1_t - dK^2_t\right) \\
		&= \mathds{1}_{\{Y^1_t \neq Y^2_t\}} \left(Y^1_t - Y^2_t\right) \left(dK^1_t + k\left(t, \frac{Y^1_t + Y^2_t}{2}\right) dt \right) \\
		&\quad + \mathds{1}_{\{Y^1_t \neq Y^2_t\}} \left(Y^2_t - Y^1_t\right) \left(dK^2_t + k\left(t, \frac{Y^1_t + Y^2_t}{2}\right) dt \right),
	\end{split}
\end{equation*}
completes the proof.
\end{proof}
As a consequence of Lemma \ref{lemma1}, we derive the following result.
\begin{corollary}\label{coro1}
	Let the quadruplets $(Y^1, Z^1, \psi^1, K^1)$ and $(Y^2, Z^2, \psi^2, K^2)$ be two solutions of the MBSDE$(\xi_1, f_1, k_1)$ and MBSDE$(\xi_2, f_2, k_2)$, respectively. If, for every $(t, x) \in [0,T] \times \mathbb{R}$, the following conditions hold:
	\begin{itemize}
		\item $a^1_t \leq a^2_t$; 
		\item $k_1(t,x) \geq k_2(t,x)$ on $]a^2_t; +\infty[$,
	\end{itemize}
	then the measure $\mathds{1}_{\{Y^1_t > Y^2_t\}}(dK^1_t - dK^2_t)$ is a.s. negative on $[0,T]$.
\end{corollary}
\begin{proof}
	On the set $\{Y^1_t > Y^2_t\}$, we have $(Y^1_t + Y^2_t)/2 \in ]a^2_t; +\infty[ \subset ]a^1_t; +\infty[$. By applying the same argument as in the Lemma \ref{lemma1}, we obtain:
	\begin{equation*}
		\begin{split}
			&\mathds{1}_{\{Y^1_t > Y^2_t\}} (Y^1_t - Y^2_t)(dK^1_t - dK^2_t) \\
			&= \mathds{1}_{\{Y^1_t > Y^2_t\}} (Y^1_t - Y^2_t)\left(dK^1_t + k_1\left(t, \frac{Y^1_t + Y^2_t}{2}\right)dt\right) \\
			&\quad + \mathds{1}_{\{Y^1_t > Y^2_t\}} (Y^1_t - Y^2_t)\left(k_2\left(t, \frac{Y^1_t + Y^2_t}{2}\right)-k_1\left(t, \frac{Y^1_t + Y^2_t}{2}\right) \right)dt \\
			&\quad +\mathds{1}_{\{Y^1_t > Y^2_t\}} (Y^2_t-Y^1_t)\left(k_2\left(t, \frac{Y^1_t + Y^2_t}{2}\right)dt + dK^2_t\right),
		\end{split}
	\end{equation*}
	which is a negative measure, completing the proof of the claim.
\end{proof}

Using Lemma \ref{lemma1}, we can establish the uniqueness of the solution for MBSDEs of the form \eqref{eq1}--\eqref{eq4}.
\begin{theorem}\label{uniq}
	Under condition (A), the MBSDE$(\xi,f, k)$ admits at most one solution $(Y_t, Z_t, \psi_t, K_t)_{t \leq T}$.
\end{theorem}
\begin{proof}
	Suppose that there exist two solutions $(Y^1, Z^1, \psi^1, K^1)$ and $(Y^2, Z^2, \psi^2, K^2)$ to the given problem, and denote $\widehat{\mathcal{R}}_t  := \mathcal{R}^1_t - \mathcal{R}^2_t$ for $\mathcal{R} \in \{Y, Z, \psi\}$. By applying Itô's formula (see, e.g., \cite[Theorem 32, page 78]{bookProtter}) to $(t, x) \mapsto e^{\zeta t } x^2$ (where $\zeta$ will be chosen
	later) and to the process $\widehat{Y}$ yields
	\begin{equation}\label{eq11}
		\begin{split}
			&e^{\zeta t} |\widehat{Y}_t|^2+\zeta \int_{t}^{T}e^{\zeta s} |\widehat{Y}_s|^2 ds\\
			&=2\int_{t}^{T}e^{\zeta s} \widehat{Y}_s \left(f(s,Y^1_s, Z^1_s,\psi^1_s)-f(s,Y^2_s, Z^2_s,\psi^2_s)\right) ds-\int_{t}^{T}e^{\zeta s} |\widehat{Z}_s|^2 ds\\
			&-\int_{t}^{T}\int_{\mathcal{U}}e^{\zeta s}  |\widehat{\psi}_s(e)|^2 \pi(de)ds-2\int_{t}^{T}e^{\zeta s} \widehat{Y}_s \widehat{Z}_s dW_s+2\int_{t}^{T}e^{\zeta s} (Y^1_s-Y^2_s)(dK^1_s-dK^2_s)\\
			&-2\int_{t}^{T}\int_{\mathcal{U}}e^{\zeta s}  \left(|Y_{s-}+\widehat{\psi}_s(e)|^2-|\widehat{Y}_{s-}|^2\right) \tilde{N}(ds,de).
		\end{split}
	\end{equation}
	From the assumptions on $f$, we have
	\begin{equation}\label{used}
		\begin{split}
			2\widehat{Y}_s \left(f(s,Y^1_s, Z^1_s,\psi^1_s)-f(s,Y^2_s, Z^2_s,\psi^2_s)\right)& \leq 2C|\widehat{Y}_s|\left(|\widehat{Y}_s|+|\widehat{Z}_s|+\|\widehat{\psi}_s\|_\pi\right)\\
			&\leq 2C (1+2 C)|\widehat{Y}_s|^2+\frac{1}{2} \left(|\widehat{Z}_s|^2+\|\widehat{\psi}_s\|_\pi^2\right)
		\end{split}
	\end{equation}
Plugging \eqref{used} into \eqref{eq11}, along with the result of Lemma \ref{lemma1}, we obtain
\begin{equation*}
	\begin{split}
		&e^{\zeta t} |\widehat{Y}_t|^2 + (\zeta - 2C (1 + 2C)) \int_{t}^{T} e^{\zeta s} |\widehat{Y}_s|^2 \, ds + \frac{1}{2} \int_{t}^{T} e^{\zeta s} (|\widehat{Z}_s|^2 + \|\widehat{\psi}_s\|^2_\pi) \, ds \\
		&\leq -2 \int_{t}^{T} e^{\zeta s} \widehat{Y}_s \widehat{Z}_s \, dW_s
		-2 \int_{t}^{T} \int_{\mathcal{U}} e^{\zeta s} \left( |Y_{s-} + \widehat{\psi}_s(e)|^2 - |\widehat{Y}_{s-}|^2 \right) \tilde{N}(ds,de).
	\end{split}
\end{equation*}
The right-hand side of the inequality is clearly a martingale. By choosing $\zeta = 2C (1 + 2C)$ and taking the expectation on both sides, we obtain
$$
\mathbb{E}\left[e^{\zeta t} |\widehat{Y}_t|^2\right] + \mathbb{E}\left[\int_{t}^{T} e^{\zeta s} (|\widehat{Z}_s|^2 + \|\widehat{\psi}_s\|^2_\pi) \, ds\right] = 0.
$$
Thus, we conclude the uniqueness of the control processes $(Z, \psi)$ and that $Y^1_t = Y^2_t$ a.s.,  for each $t \in [0, T]$. Additionally, since $Y^1$ and $Y^2$ are RCLL processes, it follows that $Y^1 = Y^2$ on $[0,T]$ in the indistinguishable sense. Finally, the uniqueness of $K$ follows directly from the uniqueness of the processes $(Y, Z, \psi)$ and the forward formulation of  the BSDE \eqref{eq1}.
\end{proof}

\begin{remark}\label{importa rmq}
	It should be noted that, in the case where, for every $t$, the function $k(t, \cdot)$ is defined and Lipschitz continuous on $\mathbb{R} = \mathcal{D}_t$, with $\int_{0}^{T} (k(s,0))^2 \, ds < +\infty$, the solution of the MBSDE $(\xi, f, k)$ is given by
	$$
	Y_t = \xi + \int_t^T \left\{f(s, Y_s, Z_s, \psi_s) - k(s,Y_s)\right\} ds - \int_t^T Z_s\, dW_s - \int_t^T \int_{\mathcal{U}} \psi_s(e)\, \tilde{N}(ds,de).
	$$
	Indeed, in such a case, we have
	$$
	k_t(x) = k(t,x) \quad \text{and} \quad 
	\operatorname{Gr}(k_t) = \left\{ (x, y) \in \mathbb{R} \times \mathbb{R} \;:\; y = k(t,x) \right\}=\operatorname{Gr}(k(t,\cdot)),
	$$
	which corresponds to a single-valued operator. Then, using Lemma 2.4 in \cite{TangLi1994}, or Theorem 3.1.1 in \cite[page 49]{Delong2014}, or Theorem 55.1 in \cite{Pardoux1997}, we deduce that there exists a unique solution $(Y, Z, \psi)$ to the above BSDE associated with $(\xi,f(t,y,z,\psi)-k(t,y))$ such that
	$$
	\mathbb{E}\left[\sup_{0 \leq t \leq T} |Y_t|^2 + \int_{0}^{T} \left( |Z_s|^2 + \|\psi_s\|^2_\pi \right) ds \right] < +\infty.
	$$
	On the other hand, by setting $K_t = -\int_{0}^{t} k(s, Y_s)\, ds$ for $t \in [0,T]$, it follows from the fact that, for every pair of optional processes $(\alpha_t, \beta_t)$ with values in $\operatorname{Gr}(k_t)$ for $t \in [0,T]$, we have $\beta_t = k(t, \alpha_t)$. Hence, the measure
	$$
	(Y_t - \alpha_t)(dK_t + \beta_t dt) = (Y_t - \alpha_t)(k(t, \alpha_t) - k(t, Y_t))\, dt
	$$
	is clearly a.s. negative on $[0,T]$ due to the monotonicity of the functions $k(t,\cdot)$. Thus, $K$ satisfies conditions \eqref{eq3}--\eqref{eq4}. Therefore, from the uniqueness result stated in Theorem \ref{uniq}, we conclude that $(Y, Z, \psi, K)$ is the unique solution of the MBSDE $(\xi, f, k)$.
\end{remark}

Let us now state our comparison results concerning MBSDEs associated with equations \eqref{eq1}--\eqref{eq4}.
\begin{theorem}\label{thm1}
	Assume that condition (A) holds.
	Let $(Y^1, Z^1, \psi^1, K^1)$ and $(Y^2, Z^2, \psi^2, K^2)$ be two solutions of the MBSDE$(\xi_1, f_1, k_1)$ and MBSDE$(\xi_2, f_2, k_2)$, respectively. If, for every $(t, x) \in [0,T] \times \mathbb{R}$, the following conditions hold:
\begin{itemize}
	\item $\xi_1 \leq \xi_2 $ a.s.
	
	\item $f_1(t,y,z,u) \leq f_2(t,y,z,u) $ a.s. for any $(t, y,z,\psi) \in [0,T] \times \mathbb{R} \times \mathbb{R} \times \mathbb{L}^2_\pi$
	
	\item $a^1_t \leq a^2_t$ and $k_1(t,x) \geq k_2(t,x)$ on $]a^2_t; +\infty[$,
\end{itemize}
Then a.s. for any $t \in [0,T]$, $Y^1_t \leq Y^2_t$.
\end{theorem}

\begin{proof}
	The proof follows closely the argument of Theorem 2.6 in \cite{royer2006backward} or . 
	We define
	\begin{equation*}
		\left\lbrace 
		\begin{split}
			h_s&:=\frac{f_1(s,Y^1_s,Z^1_s,\psi^1_s)-f_1(s,Y^2_s,Z^1_s,\psi^1_s)}{Y^1_s-Y^2_s} \mathds{1}_{\{Y^1_s \neq Y^2_s\}}\\
			g_s&:=\frac{f_1(s,Y^2_s,Z^1_s,\psi^1_s)-f_1(s,Y^2_s,Z^2_s,\psi^1_s)}{Z^1_s-Z^2_s} \mathds{1}_{\{Y^1_s \neq Y^2_s\}}\\
			f_s&:=f_1(s,Y^2_s,Z^2_s,\psi^2_s)-f_2(s,Y^2_s,Z^2_s,\psi^2_s)
		\end{split}
		\right. 
	\end{equation*}
	We then define the following new equivalent probability measure
	$$
	\dfrac{d \mathbb{Q}}{d \mathbb{P}}:=\mathscr{E}\left(\int_{0}^{T} g_s dW_s+\int_{0}^{T} \int_{\mathcal{U}} \kappa^{Y^2_s, Z^2_s, \psi^1_s, \psi^2_s}_s(e) \tilde{N}(ds,de) \right),
	$$
	where $\mathscr{E}(\mathcal{X})$ is the Doleans-Dade stochastic exponential of a given $\mathbb{F}$-semimartingale $\mathcal{X}$ (see \cite[Ch II. Section 8]{bookProtter}). Using \cite[Theorem 2.5.1 , page 28]{Delong2014}, we derive that
	$$
	{W}^\mathbb{Q}_t:=W_t-\int_{0}^{t} g_s ds ,\quad t \in [0,T]
	$$
	\text{ and } 
	$${\tilde{N}}^\mathbb{Q}((0,t] \times \Lambda):={{N}}((0,t] \times \Lambda)-\int_{0}^{t} \int_\mathcal{U}\left(1+\kappa^{Y^2_s, Z^2_s, \psi^1_s, \psi^2_s}_s(e)\right)\pi(de)ds,\quad t \in [0,T],~ \Lambda \in \mathscr{B}(\mathcal{U}),
	$$
	are respectively a $(\mathbb{F},\mathbb{Q})$-Brownian motion and a $(\mathbb{F},\mathbb{Q})$-compensated Poisson random measure.
	
	Let use set $\widehat{\xi}=\xi_1-\xi_2$ and  $\widehat{\mathfrak{R}}=\mathfrak{R}^1-\mathfrak{R}^2$ for $\mathfrak{R} \in \{Y,Z,\psi\}$. 
	Meyer-Ito Formula formula (\cite[Theorem 66., page 210]{bookProtter}) applied to the process $\widehat{Y}_t = Y^1_t - Y^2_t$ and to the convex function $x \mapsto x^+$ under the new probability measure $\mathbb{Q}$ along with the monotonicity of $f_1$ w.r.t. $\psi$ yields:

	 \begin{equation*}
	 	 \begin{split}
	 	\widehat{Y}^+_t &\leq \widehat{\xi}^+ + \int_{t}^{T} \mathds{1}_{\{\widehat{Y}_{s-}>0\}}\left(f_s+h_s \widehat{Y}_s+g_s \widehat{Z}_s+\int_{\mathcal{U}}\widehat{\psi}_s(e) \kappa^{Y^2_s, Z^2_s, \psi^1_s, \psi^2_s}_s(e) \pi({d} e)\right)ds\\
	 	&\quad+ \int_{t}^{T} \mathds{1}_{\{{Y}^1_{s}>Y^2_s\}}\left(dK^1_s-dK^2_s\right)-\int_{t}^{T} \mathds{1}_{\{\widehat{Y}_{s-}>0\}} \widehat{Z}_s dW_s- \int_t^T\int_{\mathcal{U}} \widehat{\psi}_s(e)\, \tilde{N}(ds,de)\\
	 	&= \widehat{\xi}^+ + \int_{t}^{T} \mathds{1}_{\{\widehat{Y}_{s}>0\}}\left(f_s+h_s \widehat{Y}_s\right)ds+ \int_{t}^{T} \mathds{1}_{\{{Y}^1_{s}>Y^2_s\}}\left(dK^1_s-dK^2_s\right)- \int_{t}^{T}d\mathfrak{M}^\mathbb{Q}_s
	 		  \end{split}
	 \end{equation*}
 where $d\mathfrak{M}^\mathbb{Q}_s=\mathds{1}_{\{\widehat{Y}_{s-}>0\}} \widehat{Z}_s dW^\mathbb{Q}_s+\int_{\mathcal{U}} \widehat{\psi}_s(e)\, \tilde{N}^\mathbb{Q}(ds,de)$ is an $(\mathbb{F},\mathbb{Q})$-martingale starting from zero.\\
 Using Corollary \ref{coro1} and taking the expectation on both sides of the above inequality along with the Lipschitz property of $f$ and assumptions on the driver of Theorem \ref{thm1}, we get
 $$
 \mathbb{E}_\mathbb{Q}\left[\widehat{Y}^+_t\right] \leq C \int_{t}^{T} \mathbb{E}_\mathbb{Q}\left[\widehat{Y}^+_s\right] ds,\quad t \in [0,T]
 $$
 Finally, by applying Gronwall's lemma, we get the want result under $\mathbb{Q}$ and then under $\mathbb{P}$ by equivalence.
\end{proof}

\section{Existence result}
\label{sec2}
In this section, we aim to establish the existence result for the MBSDE$(\xi, f, k)$ using a penalization approximation under assumption (A), along with two additional assumptions concerning the local-in-time integrability of the functions $k(\cdot,x)$. More precisely, we introduce the following local assumptions (B):
\begin{itemize}
	\item[(B.1)] For any $y \in ]\sup_{t \in {I}} a_t; +\infty[$, $\int_{I}|k(s,y)|ds<+\infty$.
	
	\item[(B.2)] There exists some $z \in ]\sup_{t \in {I}}a_t; +\infty[$ such that $\int_I |k(s,z)|^2 ds<+\infty$.
\end{itemize}

The following is an obvious yet important remark concerning the local assumptions (B) stated above:
\begin{remark}\label{rmq1}
	Since the functions $k(t,\cdot)$ take values in $\mathbb{R}_-$, under assumption (B) we obtain, for any compact set ${I}$ of $[0,T]$, that:
	\begin{itemize}
		\item[\text{(b.1)}] Assumption (B.1) implies that for any $y \in ]\sup_{t \in {I}} a_t; +\infty[$, $\int_{I}|k(s,y)|ds<+\infty$.
		
		\item[\text{(b.2)}] Assumption (B.2) implies that there exists some $z \in ]\sup_{t \in {I}}a_t; +\infty[$ such that $\int_I |k(s,z)|^2 ds<+\infty$.
	\end{itemize}
	These properties follow from extracting a finite sub-cover from the covering of $I$ by intervals of the form $]u, v[$, and by exploiting the decreasing nature of the functions $|k(s,\cdot)|$ and $(k(s,\cdot))^2$.
\end{remark}

The main result of this first part is given as follows:
\begin{theorem}\label{main1}
	Under the assumptions (A) and (B), there exists a unique solution $(Y_t, Z_t, \psi_t, K_t)_{t \leq T}$ to the MBSDE$(\xi,f,k)$.
\end{theorem}

Let us first state a lemma that will be used to construct the solution of our MBSDE $(\xi, g, k)$ in the case where the graphs of $k_t(\cdot)$ are contained in $\mathbb{R} \times \mathbb{R}_-$. This result is explicitly stated in \cite[page 122]{marois1990equations} and also appears in a similar form in \cite[page 523]{lepingle2006equations}.
\begin{lemma}\label{apprx lemma}
Since the functions $k(t, \cdot)$ take negative values, there exists a sequence of numerical functions $\{k_n(t,\cdot)\}_{n \geq 1}$ defined on $\mathbb{R}$ such that:
	\begin{itemize}
		\item[\text{(i)}] The functions $k_n(t, \cdot)$ take negative values, are increasing, and are Lipschitz continuous with Lipschitz constant $n$;
		
		\item[\text{(ii)}] The sequence $\{k_n(t,\cdot)\}_{n \geq 1}$ is decreasing;
		
		\item[\text{(iii)}] The sequence $\{k_n(t,\cdot)\}_{n \geq 1}$ is unbounded below for $x \in \mathcal{D}_t^c = \mathbb{R} \setminus \mathcal{D}_t$, and it converges to $k(t, x)$ when $x \in \mathcal{D}_t$, meaning:
		\begin{itemize}
			\item For $x \in \mathcal{D}_t$, we have $\lim\limits_{n \rightarrow +\infty} k_n(t,x)=k(t,x)$;
			
			\item For $x \notin \mathcal{D}_t$, we have $\lim\limits_{n \rightarrow +\infty} k_n(t,x)=-\infty$;
		\end{itemize}
		
		\item[\text{(iv)}] The functions $k_n(\cdot, x)$ are measurable on $[0,T]$.
	\end{itemize}
\end{lemma}

The explicit construction of such a sequence can be obtained by considering, for each $x$, the value $k_n(t, x)$ as the ordinate of the intersection point between $\mathrm{Gr}\, (k_t)$ and the line of slope $(-n)$ passing through the point $(x, 0)$. Moreover, the measurability of the functions $k_n(\cdot, x)$ on $[0,T]$ follows from the measurability of $k(\cdot, x)$ (see \cite{lepingle2006equations,marois1990equations} for more details).

\begin{proof}
The proof is carried out in six steps.
	\paragraph*{Step 1: Construction of the approximating MBSDEs.}
		\emph{}\\
Let $\{k_n(t,\cdot)\}_{n \in \mathbb{N}}$ be the sequence of approximation functions given by Lemma \ref{apprx lemma}. For each $t \in [0,T]$, let us consider the following BSDE with jumps:
\begin{equation}\label{approxequa}
	Y^n_t = \xi + \int_t^T \left\{f(s, Y^n_s, Z^n_s, \psi^n_s)-k_n(s,Y^n_s)\right\}\, ds - \int_t^T Z^n_s\, dW_s - \int_t^T\int_{\mathcal{U}} \psi^n_s(e)\, \tilde{N}(ds,de).
\end{equation}
We define the continuous adapted increasing process $K^n_t := -\int_{0}^{t} k_n(s,Y^n_s) \, ds$ for all $t \in [0,T]$. It follows that $K^n_0 = 0$ and $K^n_T - K^n_t = -\int_{t}^{T} k_n(s,Y^n_s) \, ds$.

Let us define the driver $f_n(\omega,t,y,z,\psi) := f(\omega, t,y,z,\psi) - k_n(t,y)$. From Lemma \ref{apprx lemma}-(i), we know that the functions $k_n(t, \cdot)$ are Lipschitz continuous with Lipschitz constant $n$. Combined with the Lipschitz property of $f$ with constant $C$, we deduce that $f_n$ is Lipschitz with respect to $(y,z,\psi)$ uniformly in $(\omega,t)$ with Lipschitz constant $n \vee C$. Additionally, from Remark \eqref{rmq1}-(b.2), we deduce that there exists some $z \in \, ]\sup_{t \in [0,T]} a_t, +\infty[$ such that $\int_0^T (k(s,z))^2 \, ds < +\infty$. Clearly, we have $z \in \bigcap_{t \in [0,T]} \mathcal{D}_t^\circ$, so we can dominate $(k_n(s,z))^2$ by $(k(s,z))^2$ for every $s \in [0,T]$, by virtue of Lemma \ref{apprx lemma}-(ii)-(iii). Thus, we obtain
\begin{equation*}
	\begin{split}
		\mathbb{E}\left[\int_{0}^{T} \left|f_n(t,0,0,0)\right|^2 dt\right] & \leq 2 \left(\mathbb{E}\left[\int_{0}^{T} \left|f(t,0,0,0)\right|^2 dt\right] + \int_{0}^{T} (k_n(t,0))^2 dt \right)\\
		& \leq  2 \left(\mathbb{E}\left[\int_{0}^{T} \left|f(t,0,0,0)\right|^2 dt\right] + 2n^2 z^2 T + 2\int_{0}^{T} (k_n(t,z))^2 dt \right)\\
		& \leq  2 \left(\mathbb{E}\left[\int_{0}^{T} \left|f(t,0,0,0)\right|^2 dt\right] + 2n^2 z^2 T + 2\int_{0}^{T} (k(t,z))^2 dt \right)\\
		& < +\infty.
	\end{split}
\end{equation*}
Using Theorem 3.1.1 in \cite[page 49]{Delong2014}, we conclude that, for each $n \geq 1$, there exists a unique solution $(Y^n_t, Z^n_t, \psi^n_t)_{t \leq T}$ of the BSDE \eqref{approxequa} associated with $(\xi, f_n)$ such that
$$
\mathbb{E}\left[\sup_{0 \leq t \leq T} |Y^n_t|^2 + \int_{0}^{T} \left( |Z^n_s|^2 + \|\psi^n_s\|^2_\pi \right) ds \right] < +\infty.
$$

	\paragraph*{Step 2: Uniform estimation of the sequence $\{Y^n, Z^n, \psi^n, K^n\}_{n \geq 1}$.}
		\emph{}\\
Let $z,\zeta \in \mathbb{R}$ be two arbitrary constants that will be chosen later. As in Theorem \ref{uniq}, by applying It\^o's formula to the function $(t,x) \mapsto e^{\zeta t} \left(x-z\right)^2$ with the process $Y^n$, we obtain
	\begin{equation}\label{eq5}
	\begin{split}
		&e^{\zeta t} (Y^n_t-z)^2+\zeta \int_{t}^{T}e^{\zeta s} (Y^n_s-z)^2 ds+\int_{t}^{T}e^{\zeta s} (|{Z}^n_s|^2+\|\psi^n_s\|^2_\pi) ds\\
		&=e^{\zeta T} (\xi-z)^2+2\int_{t}^{T}e^{\zeta s} ({Y}^n_s-z) \left(f(s,Y^n_s, Z^n_s,\psi^n_s)-k_n(s,Y^n_s)\right)ds-(\mathfrak{M}^n_T-\mathfrak{M}^n_t)
	\end{split}
\end{equation}
where $\mathfrak{M}^n$ is a martingale starting from zero.\\
Using the Lipschitz condition of $f$, the monotonicity of the function $k_n(t,\cdot)$, and choosing a constant $z$ as in Remark \ref{rmq1}-(b.2) applied to $[0, T]$, we obtain
	\begin{equation}\label{gen}
	\begin{split}
		&2 ({Y}^n_s-z) \left(f(s,Y^n_s, Z^n_s,\psi^n_s)-k_n(s,Y^n_s)\right)\\
		&=2({Y}^n_s-z) \left(f(s,Y^n_s, Z^n_s,\psi^n_s)-f(s,z, Z^n_s,\psi^n_s)\right)+2 ({Y}^n_s-z) \left(f(s,z, Z^n_s,\psi^n_s)-f(s,z, 0,0)\right)\\
		&+2 ({Y}^n_s-z) \left(f(s,z, 0,0)-f(s,0, 0,0)\right)+2 ({Y}^n_s-z) f(s,0, 0,0)\\
		&+ 2 ({Y}^n_s-z) \left(k_n(s,z)-k_n(s,Y^n_s)\right)-2({Y}^n_s-z)k_n(s,z)\\
		&\leq (2+2C+5C^2)({Y}^n_s-z)^2+\frac{1}{2}\left(|Z^n_s|^2+\|\psi^n\|^2_\pi\right)+z^2+|f(s,0,0,0)|^2+(k_n(s,z))^2\\
		&\leq (2+2C+5C^2)({Y}^n_s-z)^2+\frac{1}{2}\left(|Z^n_s|^2+\|\psi^n\|^2_\pi\right)+z^2+|f(s,0,0,0)|^2+(k(s,z))^2
	\end{split}
\end{equation}
Plugging this into \eqref{eq5}, choosing $\zeta > (2 + 2C + 5C^2)$, and taking the expectation on both sides, we obtain the existence of a constant $\mathfrak{C}_{C}$ such that
\begin{equation}\label{ui1}
	\begin{split}
		&\mathbb{E}\left[\int_{0}^{T} e^{\zeta s} (Y^n_s-z)^2 \, ds\right] + \mathbb{E}\left[ \int_{0}^{T} e^{\zeta s} (|{Z}^n_s|^2 + \|{\psi}^n_s\|^2_\pi) \, ds \right] \\
		&\leq\mathfrak{C}_{C} \left(\mathbb{E}\left[e^{\zeta T} |\xi|^2\right]+\mathbb{E}\left[\int_{0}^{T}e^{\zeta s}|f(s,0,0,0)|^2 ds\right]+z^2 T +\int_{0}^{T}e^{\zeta s}(k(s,z))^2ds\right).
	\end{split}
\end{equation}
Now, by rewriting It\^o's formula in an alternative form, we obtain
	\begin{equation*}
	\begin{split}
		e^{\zeta t} (Y^n_t-z)^2=&e^{\zeta T} (\xi-z)^2 -\zeta \int_{t}^{T}e^{\zeta s} (Y^n_s-z)^2 ds-\int_{t}^{T}e^{\zeta s} |{Z}^n_s|^2 ds\\
		&+2\int_{t}^{T}e^{\zeta s} ({Y}^n_s-z) \left(f(s,Y^n_s, Z^n_s,\psi^n_s)-k_n(s,Y^n_s)\right)ds\\
		&-\int_{t}^{T} \int_{\mathcal{U}} e^{\zeta s}|\psi^n_s(e)|^2 \pi(ds,de)-2\int_{t}^{T}e^{\zeta s} ({Y}^n_s-z) Z^n_s dW_s\\
		&-2\int_{t}^{T} \int_{\mathcal{U}} e^{\zeta s} ({Y}^n_{s-}-z) \psi^n_s(e) \tilde{N}(ds,de)
	\end{split}
\end{equation*}
which, together with \eqref{gen} and the same choice of the constant $z \in \bigcap_{t \in [0,T]} \mathcal{D}_t^\circ$ as above, gives
	\begin{equation}\label{eq6}
	\begin{split}
		e^{\zeta t} (Y^n_t-z)^2\leq&e^{\zeta T} (\xi-z)^2 +\frac{1}{2}\int_{t}^{T}e^{\zeta s} \|\psi^n_s\|^2_\pi ds+z^2(T-t)+\int_{t}^{T}e^{\zeta s} |f(s,0,0,0)|^2 ds\\
		&+\int_{t}^{T}e^{\zeta s} (k(s,z))^2 ds-2\int_{t}^{T}e^{\zeta s} ({Y}^n_s-z) Z^n_s dW_s\\
		&-2\int_{t}^{T} \int_{\mathcal{U}} e^{\zeta s} ({Y}^n_{s-}-z) \psi^n_s(e) \tilde{N}(ds,de).
	\end{split}
\end{equation}
Next, we apply the Burkholder-Davis-Gundy (BDG) inequality to obtain
\begin{equation}\label{eq7}
		\begin{split}
 &\mathbb{E}\left[\sup_{0 \leq t \leq T} \left|\int_{t}^{T}e^{\zeta s} ({Y}^n_s-z) Z^n_s dW_s\right| \right]\\
 &\leq c \mathbb{E}\left[\left(\int_{0}^{T}e^{2\zeta s} ({Y}^n_s-z)^2 |Z^n_s|^2 ds\right)^{\frac{1}{2}} \right]\\
 &\leq \frac{1}{8} \mathbb{E}\left[\sup_{0 \leq t \leq T}e^{\zeta t} ({Y}^n_t-z)^2\right]+2 c^2 \mathbb{E}\left[\int_{0}^{T}e^{\zeta s} |{Z}^n_s|^2 ds\right].
 	\end{split}
\end{equation}
By the same arguments, we have
\begin{equation}\label{eq8}
	\begin{split}
		&\mathbb{E}\left[\sup_{0 \leq t \leq T} \left|\int_{t}^{T} \int_{\mathcal{U}} e^{\zeta s} ({Y}^n_{s-}-z) \pi^n_s(e) \tilde{N}(ds,de)\right| \right]\\
		&\leq c \mathbb{E}\left[\left(\int_{0}^{T} \int_{\mathcal{U}} e^{2\zeta s} ({Y}^n_{s-}-z)^2 |\psi^n_s(e)|^2 N(ds,de)\right)^{\frac{1}{2}} \right]\\
		&\leq \frac{1}{8} \mathbb{E}\left[\sup_{0 \leq t \leq T}e^{\zeta t} ({Y}^n_t-z)^2\right]+2 c^2 \mathbb{E}\left[\int_{0}^{T}e^{\zeta s} \|\psi^n_s\|^2_\pi ds\right].
	\end{split}
\end{equation}
We shall now return to \eqref{eq6} and employ estimates \eqref{ui1}, \eqref{eq7}, and \eqref{eq8} to deduce that there exists a constant $\mathfrak{C}_{C}$ such that 
\begin{equation}\label{ui2}
	\begin{split}
		&\mathbb{E}\left[\sup_{0 \leq t \leq T} e^{\zeta t} (Y^n_t-z)^2 \right] \\
		&\leq\mathfrak{C}_{C} \left(\mathbb{E}\left[e^{\zeta T} |\xi|^2\right]+\mathbb{E}\left[\int_{0}^{T}e^{\zeta s}|f(s,0,0,0)|^2 ds\right]+z^2 T +\int_{0}^{T}e^{\zeta s}(k(s,z))^2ds\right).
	\end{split}
\end{equation}
As a particular consequence of the uniform estimates \eqref{ui1} and \eqref{ui2}, we deduce the existence of a constant $\mathfrak{C}_{C,T}$ such that
\begin{equation}\label{ui3}
	\begin{split}
			&\sup_{n \geq 1} \left\{\mathbb{E}\left[\sup_{0 \leq t \leq T} |Y^n_t|^2 \right]+\mathbb{E}\left[\int_{0}^{T}  |Y^n_s|^2 \, ds\right] + \mathbb{E}\left[ \int_{0}^{T} (|{Z}^n_s|^2 + \|{\psi}^n_s\|^2_\pi) \, ds \right] \right\}\\
			&\leq\mathfrak{C}_{C,T} \left(\mathbb{E}\left[ |\xi|^2\right]+\mathbb{E}\left[\int_{0}^{T}|f(s,0,0,0)|^2 ds\right]+z^2+\int_{0}^{T}(k(s,z))^2ds\right).
		\end{split}
\end{equation}
Finally, by rewriting the BSDE \eqref{approxequa} in forward form, we obtain
$$
K^n_t=Y^n_0-Y^n _t-\int_{0}^{t}f(s, Y^n_s, Z^n_s, \psi^n_s)ds+\int_{0}^{t} Z^n_s dW_s+\int_{0}^{t} \int_{\mathcal{U}}\psi^n_s(e)\tilde{N}(ds,de),\quad t \in [0,T].
$$
By squaring, taking expectations, and using the Lipschitz property of $f$, Hölder's inequality, and Itô's isometry, we obtain
\begin{equation*}
	\begin{split}
		&\mathbb{E}\left[|K^n_T|^2\right]\\
		& \leq \mathfrak{C}_{C,T}\left(\mathbb{E}\left[\sup_{0 \leq t \leq T}|Y^n_t|^2\right]+\mathbb{E}\left[ \int_{0}^{T} (|{Z}^n_s|^2 + \|{\psi}^n_s\|^2_\pi) \, ds \right]+\mathbb{E}\left[\int_{0}^{T}|f(s,0,0,0)|^2 ds\right]\right).
	\end{split}
\end{equation*}
This, together with \eqref{ui3}, allows us to obtain
\begin{equation}\label{ui4}
	\begin{split}
		\sup_{n \geq 1}\left\{\mathbb{E}\left[|K^n_T|^2\right]\right\}
		&\leq\mathfrak{C}_{C,T} \left(\mathbb{E}\left[ |\xi|^2\right]+\mathbb{E}\left[\int_{0}^{T}|f(s,0,0,0)|^2 ds\right]+z^2+\int_{0}^{T}(k(s,z))^2ds\right).
	\end{split}
\end{equation}

	\paragraph*{Step 3: Construction of the state variable $Y=(Y_t)_{t \leq T}$.}
		\emph{}\\
From Remark \ref{importa rmq}, we observe that the quadruplet $(Y^n, Z^n, \psi^n, K^n)$ solves the MBSDE$(\xi, f, k_n)$. Since $k_{n+1}(t,x) \leq k_n(t,x)$ on $]-\infty, +\infty[$ by Lemma \ref{apprx lemma}, it follows from the comparison Theorem \ref{thm1} that the sequence $\{Y^n\}_{n \geq 1}$ is increasing; that is, a.s., for all $t \in [0,T]$, we have $Y^{n+1}_t \geq Y^n_t$. Consequently, there exists a right lower semi-continuous process $Y=(Y_t)_{t \leq T}$ such that $\mathbb{P}$-a.s., for all $t \in [0,T]$, $Y_t=\lim\limits_{n \rightarrow +\infty} Y^n_t$.\\
Moreover, by applying Fatou's lemma to \eqref{ui3}, together with the Dominated Convergence Theorem, we obtain
\begin{equation}\label{cv1}
	\lim\limits_{n \rightarrow +\infty} \mathbb{E}\left[\int_{0}^{T}  |Y^n_s-Y_s|^2 \, ds\right]=0.
\end{equation}

	\paragraph*{Step 4: Construction of the control variables $(Z,\psi,K)$.}
		\emph{}\\
Now let $n \geq m \geq 1$, and consider the quadruples $(Y^n, Z^n, \psi^n, K^n)$ and $(Y^m, Z^m, \psi^m, K^m)$ to be the solutions of the MBSDE$(\xi, f, k_n)$ and MBSDE$(\xi, f, k_m)$, respectively. By applying It\^o's formula to the function $x \mapsto e^{\zeta t} x^2$ with the process $Y^n - Y^m$ and taking the expectation, we have
	\begin{equation}\label{eq9}
	\begin{split}
		& \mathbb{E}\left[|Y^n_t-Y^m_t|^2\right]+\mathbb{E}\left[ \int_{t}^{T} (|{Z}^n_s-Z^m_s|^2+\|{\psi}^n_s-{\psi}^m_s\|^2_\pi) ds\right] \\
		&=2\mathbb{E}\left[\int_{t}^{T}({Y}^n_s-{Y}^m_s) \left(f(s,Y^n_s, Z^n_s,\psi^n_s)-f(s,Y^m_s, Z^m_s,\psi^m_s)\right) ds\right] \\
		&+2\mathbb{E}\left[\int_{t}^{T}(Y^n_s-Y^m_s)(dK^n_s-dK^m_s)\right]. 
	\end{split}
\end{equation}
Using the monotonicity of the functions $k_n(t,\cdot)$ and the definition of $K^n$, along with the fact that $Y^n \geq Y^m$ for every $n \geq m$, we get
\begin{equation}\label{ret}
	\begin{split}
		&(Y^n_s-Y^m_s)(dK^n_s-dK^m_s)\\
		& \leq (Y^n_s-Y^m_s)dK^n_s\\
		&=(Y^m_s-Y^n_s)k_n(s,Y^n_s)ds\\
		&\leq (Y^m_s-Y^n_s)k_n(s,Y^1_s)ds\\
		&\leq (Y^m_s-Y^n_s)k_n(s,Y^1_s)\mathds{1}_{\{Y^1_s>z\}} ds+(Y^m_s-Y^n_s)k_n(s,Y^1_s)\mathds{1}_{\{Y^1_s \leq z\}} ds\\
		&\leq (Y^m_s-Y^n_s)k_n(s,z)\mathds{1}_{\{Y^1_s>z\}} ds+(Y^m_s-Y^n_s)k_n(s,Y^1_s)\mathds{1}_{\{Y^1_s \leq z\}} ds.
		\end{split}
\end{equation}
Additionally, for any arbitrary positive, adapted, and RCLL process $(\varepsilon^z_t)_{t \leq T}$, we have
$$(Y^m_s - Y^n_s)k_n(s, Y^1_s)\mathds{1}_{\{Y^1_s \leq z\}} ds \leq (Y^m_s - Y^n_s)k_n\big( s, Y^1_s + \big( Y^1_s - \sup_{t \in [0,T]} a_t \big)^- + \varepsilon^z_s \big) \mathds{1}_{\{Y^1_s \leq z\}} ds,$$
where $x^- = - \min(x, 0)$ and $z \in ]\sup_{t \in {[0,T]}} a_t; +\infty[$ is the same element chosen above to obtain the estimates \eqref{ui3} and \eqref{ui4}. Note that here, since $Y^1_s + \big( Y^1_s - \sup_{t \in [0,T]} a_t \big)^- + \varepsilon^z_s \in ]\sup_{t \in {[0,T]}} a_t; +\infty[$, and using Remark \ref{rmq1}-(b.1), we deduce that $\int_{0}^{T} |k_n\big( s, Y^1_s + \big( Y^1_s - \sup_{t \in [0,T]} a_t \big)^- + \varepsilon^z_s \big)| ds \leq \int_{0}^{T} |k\big( s, Y^1_s + \big( Y^1_s - \sup_{t \in [0,T]} a_t \big)^- + \varepsilon^z_s \big)| ds < +\infty$ a.s. Additionally, we have $\lim\limits_{n \rightarrow +\infty} k_n\big( s, Y^1_s + \big( Y^1_s - \sup_{t \in [0,T]} a_t \big)^- + \varepsilon^z_s \big) = k\big( s, Y^1_s + \big( Y^1_s - \sup_{t \in [0,T]} a_t \big)^- + \varepsilon^z_s \big)$ from Lemma \ref{apprx lemma}-(iii). Thus, the above inequality is well-defined for any $n \geq 1$.

From Hölder's inequality, we have
\begin{equation}\label{eq10}
	\begin{split}
		&\mathbb{E}\left[\int_{0}^{T}(Y^m_s-Y^n_s)k_n(s,z)\mathds{1}_{\{Y^1_s>z\}} ds\right]\\
		& \leq \left(\mathbb{E}\left[\int_{0}^{T} |Y^m_s-Y^n_s|^2 ds\right]\right)^{\frac{1}{2}} \left(\int_{0}^{T}(k_n(s,z))^2 ds\right)^{\frac{1}{2}}\\
		& \leq \left(\mathbb{E}\left[\int_{0}^{T} |Y^m_s-Y^n_s|^2 ds\right]\right)^{\frac{1}{2}} \left(\int_{0}^{T}(k(s,z))^2 ds\right)^{\frac{1}{2}}.
	\end{split}
\end{equation}
For the remaining term, by using the same arguments, we obtain
\begin{equation*}
	\begin{split}
		&\mathbb{E}\left[\int_{0}^{T}(Y^m_s-Y^n_s)k_n(s,Y^1_s)\mathds{1}_{\{Y^1_s \leq z\}} ds\right]\\
		&\leq \mathbb{E}\left[\int_{0}^{T}(Y^m_s-Y^n_s)k_n\big( s,Y^1_s+\big( Y^1_s-\sup_{t \in [0,T]} a_t\big)^- +\varepsilon^z_s \big) \mathds{1}_{\{Y^1_s \leq z\}} ds \right]\\
		& \leq \left(\mathbb{E}\left[\int_{0}^{T} |Y^m_s-Y^n_s|^2 ds\right]\right)^{\frac{1}{2}} \left(\mathbb{E}\left[\int_{0}^{T}\big(k_n\big( s,Y^1_s+\big( Y^1_s-\sup_{t \in [0,T]} a_t\big)^- +\varepsilon^z_s \big)\big)^2 \mathds{1}_{\{Y^1_s \leq z\}} ds\right] \right)^{\frac{1}{2}}
	\end{split}
\end{equation*}
Then, on the set $\{Y^1_s \leq z\}$: on the subset $\{Y^1_s > \sup_{t \in [0,T]} a_t\}$, we choose $\varepsilon^z$ such that $Y^1_s + \varepsilon^z_s = z$; and on the subset $\{Y^1_s \leq \sup_{t \in [0,T]} a_t\}$, we choose $\varepsilon^z$ such that $\sup_{t \in [0,T]} a_t + \varepsilon^z_s = z$. Thus, on the set $\{Y^1_s \leq z\}$, by defining $
\varepsilon^z_s = (z - Y^1_s)\mathds{1}_{\{Y^1_s > \sup_{t \in [0,T]} a_t\}} + (z - \sup_{t \in [0,T]} a_t)\mathds{1}_{\{Y^1_s \leq \sup_{t \in [0,T]} a_t\}}$, we get
\begin{equation}\label{eq12}
	\begin{split}
		&\mathbb{E}\left[\int_{t}^{T}(Y^m_s-Y^n_s)k_n(s,Y^1_s)\mathds{1}_{\{Y^1_s \leq z\}} ds\right]\\
		& \leq \left(\mathbb{E}\left[\int_{0}^{T} |Y^m_s-Y^n_s|^2 ds\right]\right)^{\frac{1}{2}} \left(\int_{0}^{T}\big(k_n\big( s,z)\big)^2 ds\right)^{\frac{1}{2}}\\
			& \leq \left(\mathbb{E}\left[\int_{0}^{T} |Y^m_s-Y^n_s|^2 ds\right]\right)^{\frac{1}{2}} \left(\int_{0}^{T}\big(k\big( s,z)\big)^2 ds\right)^{\frac{1}{2}}.
	\end{split}
\end{equation}
Returning to \eqref{ret}, and using \eqref{eq10} and \eqref{eq12}, we obtain
\begin{equation*}
	\mathbb{E}\left[\int_{t}^{T}(Y^n_s-Y^m_s)(dK^n_s-dK^m_s)\right]  \leq 2 \left(\mathbb{E}\left[\int_{0}^{T} |Y^m_s-Y^n_s|^2 ds\right]\right)^{\frac{1}{2}} \left(\int_{0}^{T}\big(k\big( s,z)\big)^2 ds\right)^{\frac{1}{2}}.
\end{equation*}
This, together with the convergence result \eqref{cv1}, yields
\begin{equation}\label{cv2}
\lim\limits_{n,m \rightarrow+\infty} \sup_{0 \leq t \leq T} \mathbb{E}\left[\int_{t}^{T}(Y^n_s-Y^m_s)(dK^n_s-dK^m_s)\right]=0.
\end{equation}
Returning to \eqref{eq9} and applying the same estimates as those used in \eqref{used} for the driver $f$, we obtain
\begin{equation}\label{eq13}
	\begin{split}
		&\mathbb{E}\left[ \int_{0}^{T} (|{Z}^n_s-Z^m_s|^2 + \|{\psi}^n_s-{\psi}^m_s\|^2_\pi) \, ds \right] \\
		&\leq 2C (1+2 C)\mathbb{E}\left[\int_{0}^{T} |Y^m_s-Y^n_s|^2 ds\right] +2\mathbb{E}\left[\int_{t}^{T}(Y^n_s-Y^m_s)(dK^n_s-dK^m_s)\right].
	\end{split}
\end{equation}
Applying \eqref{cv1} and \eqref{cv2} to \eqref{eq13}, we deduce that the sequence $\{Z^n, \psi^n\}_{n \geq 1}$ is Cauchy in $\mathbb{L}^2\left(\Omega \times [0,T], d\mathbb{P} \otimes dt\right) \times \mathbb{L}^2\left(\Omega \times [0,T] \times \mathcal{U}, d\mathbb{P} \otimes dt \otimes \pi(de)\right)$. Thus, there exists a pair of processes $(Z,\psi) \in \mathcal{P} \times \left(\mathcal{P} \otimes \mathscr{B}(\mathcal{U})\right)$ such that
$$
\mathbb{E}\left[\int_{0}^{T} \left(|Z_s|^2+\|\psi_s\|^2_\pi\right)ds \right]<+\infty
$$
and
\begin{equation}\label{cv3}
	\lim\limits_{n \rightarrow +\infty}\mathbb{E}\left[\int_{0}^{T} \left(|Z^n_s-Z_s|^2+\|\psi^n_s-\psi_s\|^2_\pi\right)ds \right]=0.
\end{equation}
Now, by applying It\^o's formula, we obtain
\begin{equation*}
	\begin{split}
	 |Y^n_t-Y^m_t|^2=&-\int_{t}^{T} |{Z}^n_s-Z^m_s|^2 ds+2\int_{t}^{T}(Y^n_s-Y^m_s)(dK^n_s-dK^m_s)\\
		&+2\int_{t}^{T} ({Y}^n_s-Y^m_s) \left(f(s,Y^n_s, Z^n_s,\psi^n_s)-f(s,Y^m_s, Z^m_s,\psi^m_s)\right)ds\\
		&-\int_{t}^{T} \int_{\mathcal{U}} |\psi^n_s(e)-\psi^m_s(e)|^2 \pi(ds,de)-2\int_{t}^{T} ({Y}^n_s-Y^m_s) (Z^n_s-Z^m_s) dW_s\\
		&-2\int_{t}^{T} \int_{\mathcal{U}} ({Y}^n_{s-}-Y^m_{s-}) (\psi^n_s(e)-\psi^m_s(e)) \tilde{N}(ds,de).
	\end{split}
\end{equation*}
Taking the supremum on both sides and applying the BDG inequality as in \eqref{eq7} and \eqref{eq8}, we deduce the existence of a constant $\mathfrak{C}_{C,T}$ such that
\begin{equation*}
	\begin{split}
	&\mathbb{E}\left[\sup_{0 \leq t \leq T}|Y^n_t-Y^m_t|^2\right] \\
	&\leq \mathfrak{C}_{C,T}\left(\mathbb{E}\left[ \int_{0}^{T} (|{Y}^n_s-Y^m_s|^2+|{Z}^n_s-Z^m_s|^2 + \|{\psi}^n_s-{\psi}^m_s\|^2_\pi) \, ds \right] \right.\\
	&\left.\qquad+\left(\mathbb{E}\left[\int_{0}^{T} |Y^m_s-Y^n_s|^2 ds\right]\right)^{\frac{1}{2}} \left(\int_{0}^{T}\big(k\big( s,z)\big)^2 ds\right)^{\frac{1}{2}}\right).
	\end{split}
\end{equation*}
Using \eqref{cv1}, \eqref{cv2}, and \eqref{cv3}, it follows 
\begin{equation*}
\lim\limits_{n,m\rightarrow +\infty}\mathbb{E}\left[\sup_{0 \leq t \leq T}|Y^n_t-Y^m_t|^2\right]=0
\end{equation*}
and then
\begin{equation}\label{cv4}
	\lim\limits_{n\rightarrow +\infty}\mathbb{E}\left[\sup_{0 \leq t \leq T}|Y^n_t-Y_t|^2\right]=0.
\end{equation}
From the convergence result \eqref{cv4}, we can extract a subsequence $\{Y^{\varphi(n)}\}_{n \geq 1}$ of $\{Y^n\}_{n \geq 1}$ that converges uniformly to $Y$. Then, using Fatou's lemma and \eqref{ui3}, we infer
\begin{equation}\label{eq14}
	\begin{split}
		&\mathbb{E}\left[\sup_{0 \leq t \leq T}|Y_t|^2\right] \\
		&\leq \liminf_{n \rightarrow +\infty}\mathbb{E}\left[\sup_{0 \leq t \leq T}|Y^{\varphi(n)}_t|^2\right]\\
		&\leq\mathfrak{C}_{C,T} \left(\mathbb{E}\left[ |\xi|^2\right]+\mathbb{E}\left[\int_{0}^{T}|f(s,0,0,0)|^2 ds\right]+z^2+\int_{0}^{T}(k(s,z))^2ds\right).
	\end{split}
\end{equation}
Now, we set, for every $t \in [0, T]$,
$$
K_T-K_t=Y_t-\xi-\int_t^T f(s, Y_s, Z_s, \psi_s)\, ds + \int_t^T Z_s\, dW_s +\int_t^T\int_{\mathcal{U}} \psi_s(e)\, \tilde{N}(ds,de)
$$
Moreover,
$$
K^n_T-K^n_t=Y^n_t-\xi - \int_t^T f(s, Y^n_s, Z^n_s, \psi^n_s) ds + \int_t^T Z^n_s\, dW_s+ \int_t^T\int_{\mathcal{U}} \psi^n_s(e)\, \tilde{N}(ds,de)
$$
By applying Doob's inequality and using the Lipschitz property of $f$, together with the convergence results \eqref{cv3} and \eqref{cv4}, we obtain
\begin{equation}\label{cv5}
	\lim\limits_{n\rightarrow +\infty}\mathbb{E}\left[\sup_{0 \leq t \leq T}|K^n_t-K_t|^2\right]=0.
\end{equation}
Therefore, from \eqref{ui4} and \eqref{cv5}, we have $\mathbb{E}\left[|K_T|^2\right] < +\infty$ with $K_0 = 0$, and there exists a subsequence $\phi(n)$ such that, for every $t \in [0, T]$, we have
\begin{equation}\label{K}
K_t=\lim\limits_{n \rightarrow +\infty}\int_{0}^{t} \left( -k_{\phi(n)}\left(s, Y^{(\phi(n))}_s\right)\right) ds,\quad \mathbb{P}\text{-a.s.}
\end{equation}

Consequently, the constructed quadruplet $(Y, Z, \psi, K)$ satisfies the following BSDE:
\begin{equation}\label{eq15}
		Y_t = \xi + \int_t^T f(s, Y_s, Z_s, \psi_s)\, ds - \int_t^T Z_s\, dW_s - \int_t^T\int_{\mathcal{U}} \psi_s(e)\, \tilde{N}(ds,de) + (K_T - K_t),\quad t \in [0,T]
\end{equation}
which corresponds to \eqref{eq1}. Additionally, the property \eqref{eq3} is satisfied by the process $K$, and by construction, we clearly have
$$
\mathbb{E}\left[\sup_{0 \leq t \leq T} |Y_t|^2 + \int_{0}^{T} \left( |Z_s|^2 + \|\psi_s\|^2_\pi \right) ds + | K_T|^2 \right] < +\infty.
$$
	\paragraph*{Step 5: $K$ is continuous and $Y$ verifies \eqref{eq2}.}
	\emph{}\\
$\bullet$ The sequence $\{Y^n\}_{n \geq 1}$ is increasing and converges to the process $Y$, which satisfies the BSDE \eqref{eq15}. Then, using Lemma 2.2 in \cite{peng1999monotonic}, we deduce that the limiting processes $Y$ and $K$ have RCLL paths on $[0,T]$. Now, from \eqref{eq15}, we have for all $t \in [0,T]$
$$
(K_t-K_{t-})=-(Y_t-Y_{t-})+\int_{\mathcal{U}} \psi_t(e)N(\{t\},de).
$$
On the other hand, from \eqref{approxequa}, we have
$$
Y^{\phi(n)}_t-Y^{\phi(n)}_{t-}=\int_{\mathcal{U}} \psi^{\phi(n)}_t(e)N(\{t\},de).
$$
From the bound of $Y^{\phi(n)}_{t}$ by $Y_{t}$, we deduce:
$$
Y_t \geq Y^{\phi(n)}_{t-}+\int_{\mathcal{U}} \psi^{\phi(n)}_t(e)N(\{t\},de).
$$
Using the convergence results \eqref{cv3} and \eqref{cv4}, and passing to the limit as $n \rightarrow +\infty$ (at least in the \textsc{UCP} sense; see, e.g., \cite[Definition on page 57]{bookProtter}), we get
$$
Y_t \geq Y_{t-}+\int_{\mathcal{U}} \psi_t(e)N(\{t\},de).
$$
Consequently, $(K_t - K_{t-})$ is negative, and since the process $K_t$ is increasing, we obtain the left-continuity of $K_t$, and hence the continuity of $K$. Thus, the constructed process $K$ satisfies \eqref{eq3}.\\

$\bullet$ Now, we can show that, for every $t \in [0,T]$, $Y_t$ takes its values in $[a_t; +\infty[ \cap \mathbb{R}$. It is therefore necessary to establish this property in the case where $a_t > -\infty$. We proceed by contradiction. If $Y_t$ is strictly less than $a_t$, then using the right-continuity of the process $Y$, there exists a real number $\varepsilon > 0$ and an interval $[t, s]$ over which $Y$ takes its values in $]-\infty, a_t - \varepsilon[$. Since the function $[0,T] \ni t \mapsto a_t$ is continuous at $t$, it is bounded from below by $(a_t - \frac{\varepsilon}{2})$ on a neighborhood of $t$. In particular, we have $a_u > a_t - \frac{\varepsilon}{2}$ for any $u \in [t,s]$. Consequently, for every $u$ in a right-hand neighborhood of $t$, we have $Y^{\phi(n)}_u \leq Y_u < a_t - \varepsilon < a_u - \frac{\varepsilon}{2}$ for all $u \in [t,s]$, and the sequence $\{k_{\phi(n)}(u, Y^{\phi(n)}_u)\}_{n \geq 1}$ is then bounded above by the sequence $\{k_{\phi(n)}(u, a_u - \frac{\varepsilon}{2})\}_{n \geq 1}$ using the fact that the functions $k_{\phi(n)}(u,\cdot)$ are increasing. Moreover, from Lemma \ref{apprx lemma}, the sequence $\{k_{\phi(n)}(u, a_u - \frac{\varepsilon}{2})\}_{n \geq 1}$ is decreasing, negative, and unbounded from below. Using \eqref{K} and the monotone convergence theorem, we get
$$
K_s - K_t = \lim\limits_{n \rightarrow +\infty} \int_{t}^{s} \left( -k_{\phi(n)}\left(u, Y^{\phi(n)}_u\right)\right)  du 
\geq 
\lim\limits_{n \rightarrow +\infty} \int_{t}^{s} \left( -k_{\phi(n)}\left(u, a_u - \frac{\varepsilon}{2}\right) \right) du = +\infty,
$$
and we thus arrive at a contradiction. As a result, the property \eqref{eq2} is verified.

\paragraph*{Step 6: Negativity of the measure $(Y_t-\alpha_t)(dK_t+\beta_t dt)$ on $[0,T]$.}
	\emph{}\\
From the upper bound of $k_{\phi(n)}(t, Y^{(\phi(n))}_t)$ by $k_{\phi(n)}(t, Y_t)$, we deduce, by integrating over the interval $[s, s'] \subset [0,T]$ and then passing to the limit (using \eqref{K} and applying the monotone convergence theorem to the sequence $\{k_{\phi(n)}(t, Y_t)\}_{n \geq 1}$), the inequality
$$
K_{s'}-K_s \geq \int_{s}^{s'}(-\lim\limits_{n \rightarrow +\infty} k_{\phi(n)}(t, Y_t)) \mathds{1}_{\{Y_t \notin \mathcal{D}_t\}}dt-\int_{s}^{s'} k(t, Y_t) \mathds{1}_{\{Y_t \in \mathcal{D}_t\}}dt.
$$
Equivalently, we can write
$$
\int_{s}^{s'}\mathds{1}_{\{Y_t \in \mathcal{D}_t\}}\left(dK_t+k(t,Y_t)dt\right)+\int_{s}^{s'}\mathds{1}_{\{Y_t \notin \mathcal{D}_t\}}\left(dK_t+\lim\limits_{n \rightarrow +\infty} k_{\phi(n)}(t, Y_t)dt\right) \geq 0.
$$
Using Lemma \ref{apprx lemma}-(iii), we know that the sequence $\{k_{\phi(n)}(t, Y_t)\}_{n \geq 1}$ is decreasing and unbounded from below for all $Y_t$ not belonging to $\mathcal{D}_t$. Thus, we get
\begin{equation}\label{m1}
	\mathds{1}_{\{Y_t \in \mathcal{D}_t\}}\left(dK_t+k(t,Y_t)dt\right) \geq 0
\end{equation}
and
\begin{equation}\label{m2}
	\mathds{1}_{\{Y_t \notin \mathcal{D}_t\}}dt=0.
\end{equation}
On the other hand, using \cite[Ch. 3 Theorem 32]{bookHe}, we deduce that the set $\{\Delta Y \neq 0\}$ is exhausted by a sequence of strictly positive stopping times. In particular, the set $\{t \in [0,T] : Y_{t-} \neq Y_t \}$ is countable. Then, using the continuity of the process $K$, we deduce that $\mathds{1}_{\{Y_t \in \mathcal{D}^\circ_t\}}(dK_t+k_{-}(t,Y_t)dt)$ is equal in measure to $\mathds{1}_{\{Y_t \in \mathcal{D}^\circ_t, Y_{t-} \in \mathcal{D}^\circ_t \}}(dK_t+k_{-}(t,Y_t)dt)$. Now, if $Y_t$ and $Y_{t-}$ belong to $\mathcal{D}^\circ_t=]a_t, +\infty[$, we can find an element $\Lambda$ in $]a_t, Y_{t-} \wedge Y_t[$. As we know that $Y^n \nearrow Y$ as  $n \rightarrow +\infty$, we may also find an interval centered around $t$, on which, from a certain rank onward, the sequence $\{Y^{(\phi(n))}\}_{ n \geq 1}$ is larger than $\Lambda$. Let $[s,s']$ be any compact subinterval of $\{t \in [0,T] : Y_{t-}  \in \mathcal{D}^\circ_t \text{ and } Y_{t}  \in \mathcal{D}^\circ_t  \}$. As $\int_{s}^{s'} |k_{-}(u,Y^{(\phi(n))}_u)| du \leq \int_{s}^{s'} |k_{-}(u,\Lambda)| du\leq \int_{s}^{s'} |k(u,\Lambda)| du$ and using Remark \ref{rmq1}-(b.1) on $[s,s']$, we get, after applying the dominated convergence theorem and using the left-continuity of the function $x \mapsto k_-(t,x)$, the limit
$$
\lim\limits_{n \rightarrow +\infty} \int_{s}^{s'} k_{-}(u,Y^{(\phi(n))}_u) du=\int_{s}^{s'} k_{-}(u,Y_u) du.
$$
From this and the inequality $k_{\phi(n)}(t,Y^{(\phi(n))}) \geq k_{-}(t,Y^{(\phi(n))})$ (for $n$ large enough), we get, for any $[s,s']$ in $\{t \in [0,T] : Y_{t-}  \in \mathcal{D}^\circ_t \text{ and } Y_{t}  \in \mathcal{D}^\circ_t  \}$,
\begin{equation*}
	\begin{split}
		\int_{s}^{s'}\mathds{1}_{\{Y_t \in \mathcal{D}^\circ_t\}}\left(dK_t+k_-(t,Y_t)dt\right)&=\int_{s}^{s'}	\mathds{1}_{\{Y_t \in \mathcal{D}^\circ_t\}}\left(k_-(t,Y_t)-\lim\limits_{n \rightarrow +\infty}k_{\phi(n)}(t,Y^{(\phi(n))}_t)\right)dt\\
		&\leq \int_{s}^{s'}	\mathds{1}_{\{Y_t \in \mathcal{D}^\circ_t\}}\left(k_-(t,Y_t)-\lim\limits_{n \rightarrow +\infty}k_{-}(t,Y^{(\phi(n))}_t)\right)dt=0.
	\end{split}
\end{equation*}
Consequently, we have
\begin{equation}\label{m3}
	\mathds{1}_{\{Y_t \in \mathcal{D}^\circ_t\}}\left(dK_t+k_-(t,Y_t)dt\right) \leq 0.
\end{equation}

Now, let us consider an arbitrary pair of optional processes $(\alpha_t,\beta_t)$ that belongs to $\operatorname{Gr}(k_t)$. The measure $(Y_t-\alpha_t)(dK_t+\beta_t dt)$ can be decomposed as follows
\begin{equation}\label{decom}
	\begin{split}
		(Y_t-\alpha_t)(dK_t+\beta_t dt)
		&= \underbrace{\mathds{1}_{\{Y_t \in \mathcal{D}^\circ_t\}}(Y_t-\alpha_t)(dK_t+\beta_t dt)}_{\mathsf{M}_1}
		+ \underbrace{\mathds{1}_{\{Y_t \notin \mathcal{D}_t\}}(Y_t-\alpha_t)(dK_t+\beta_t dt)}_{\mathsf{M}_2} \\
		&\quad + \underbrace{\mathds{1}_{\{Y_t =a_t,\, a_t \in  \mathcal{D}_t\}}(Y_t-\alpha_t)(dK_t+\beta_t dt)}_{\mathsf{M}_3}
	\end{split}
\end{equation}

$\bullet$ Due to the RCLL property of $Y$, the continuity of $K$, and the fact that the Lebesgue measure does not charge countable sets, we deduce that $\mathsf{M}_1=\mathds{1}_{\{Y_t \in \mathcal{D}^\circ_t \text{ and } Y_{t-} \in \mathcal{D}^\circ_t\}}(Y_t-\alpha_t)(dK_t+\beta_t dt)$. From \eqref{m1} and \eqref{m3}, we derive that $dK_t$ is absolutely continuous with respect to the Lebesgue measure on any interval included in the set $\{t : Y_{t-} \in \mathcal{D}^\circ_t \text{ and } Y_t \in \mathcal{D}^\circ_t \}$, with density $\Phi_t \in [-k(t,Y_t), -k_-(t,Y_t)]$ a.s. In other words, we have $(Y_t,-\Phi_t) \in \operatorname{Gr}(k_t)$ a.s. Recall that since $k_t(\cdot)$ is a maximal monotone operator, then
$$
(x,y) \in \operatorname{Gr}(k_t) \Leftrightarrow  (y-v)(x-u) \geq 0, \quad \forall (u,v) \in\operatorname{Gr}(k_t).
$$
Therefore, we derive that $(Y_t-\alpha_t)(-\Phi_t-\beta_t)dt$ is a positive measure, i.e., $\mathsf{M}_1=(Y_t-\alpha_t)(\Phi_t+\beta_t)dt =-(Y_t-\alpha_t)(-\Phi_t-\beta_t)dt$, which is a.s. negative on $[0,T]$.

$\bullet$ For $\mathsf{M}_2$, from \eqref{m2} and the fact that $Y_t$ takes values in $[a_t, +\infty[ \cap \mathbb{R}$, we deduce that $\mathsf{M}_2=\mathds{1}_{\{Y_t =a_t, \; a_t \notin \mathcal{D}_t\}}(a_t-\alpha_t)dK_t$. Since $(a_t-\alpha_t)$ takes negative values and $dK$ is an increasing measure, we obtain that $\mathsf{M}_2$ is a.s. negative in the sense of measures on $[0,T]$. 

$\bullet$ For $\mathsf{M}_3$, we can write
\begin{equation*}
	\begin{split}
		\mathsf{M}_3&=\mathds{1}_{\{Y_t =a_t,\, a_t \in  \mathcal{D}_t\}}(Y_t-\alpha_t)(dK_t+\beta_t dt)\\
		&=\underbrace{\mathds{1}_{\{Y_t =a_t,\, a_t \in  \mathcal{D}_t\}}(a_t-\alpha_t)(dK_t+k(t,Y_t) dt)}_{\mathsf{M}_{3,1}}+\underbrace{\mathds{1}_{\{Y_t =a_t,\, a_t \in  \mathcal{D}_t\}}(Y_t-\alpha_t)(\beta_t-k(t,Y_t))dt}_{\mathsf{M}_{3,2}}
	\end{split}
\end{equation*}
From \eqref{m1} and the fact that $(a_t - \alpha_t)$ takes negative values, we deduce that $\mathsf{M}_{3,1}$ is a.s. non-positive on $[0,T]$. On the other hand, using the monotonicity of the operator $k_t(\cdot)$ and the fact that $(Y_t, k(t,Y_t)) \in \operatorname{Gr}(k_t)$ on the set $\{Y_t = a_t,\, a_t \in \mathcal{D}_t\}$, we also obtain that $\mathsf{M}_{3,2}$ is a.s. non-positive on $[0,T]$.\\
Consequently, from \eqref{decom} and the non-positivity of $\mathsf{M}_1$, $\mathsf{M}_2$, and $\mathsf{M}_3$ on $[0,T]$, we obtain \eqref{eq4}.

This concludes the proof of Theorem \ref{main1}.
\end{proof}

\section{Extension to the general case where the graphs of $k_t(\cdot)$ are contained in $\mathbb{R} \times \mathbb{R}$}
\label{sec4}
In this section, we aim to discuss the extension of the existence result for a solution to the MBSDE$(\xi, f, k)$ stated in Theorem \ref{main1} to the general case of a moving barrier or boundary, where the graphs of the operators lie in $\mathbb{R} \times \mathbb{R}$. In this setting, the functions $k(t,\cdot)$ take values in $\mathbb{R}$. To proceed, we impose the following technical condition (C):
\begin{itemize}
	\item[(C)] We assume that there exists a function $\ell: [0,T] \times \mathbb{R} \to \mathbb{R}_+$ that is increasing and right-continuous in each of its variables, such that $\ell(t,x)>0$ for all $(t,x) \in [0,T) \times \mathbb{R}$, $\ell(T,x)=0$, and $|\ell(t,x)| \leq C (1+|x|)$ for all $(t,x)$. Moreover, for all $x \in \mathcal{D}_t$, we have
	$$
	\left(k(t,x)\right)^+ \leq \ell(t,x).
	$$
\end{itemize}
Assumption (C) ensures that, for every fixed $x$, the function $k(\cdot, x)$ is locally bounded from above by $\ell(\cdot,x)$ and vanishes only at the terminal time $T$. As an example, we may consider $k(t,x):=(T-t)x$ and $\ell(t,x)=(T-t)(1+x^+)$, which clearly satisfy all parts of assumption (C).

We are now ready to state our general existence result:
\begin{theorem}\label{main2}
	Under assumptions \text{(A)}, (B), and (C), there exists a unique solution $(Y_t, Z_t, \psi_t, K_t)_{t \leq T}$ to the MBSDE$(\xi,f,k)$.
\end{theorem}

In order to prove Theorem \ref{main2}, we draw inspiration from the method used in \cite[page 119]{marois1990equations}. For every $t \in [0,T]$ and each given operator $k_t(\cdot)$, we aim to construct a sequence of operators $\{k^n_t(\cdot)\}_{n \geq 1}$ with the same domain $\mathcal{D}_t$ as $k_t(\cdot)$, associated with the increasing function defined by $k_n(t,\cdot):=k(t,\cdot) \wedge n$. Clearly, the graphs of these $\{k^n_t(\cdot)\}_{n \geq 1}$ are included in $\mathbb{R} \times ] -\infty, n]$. The idea is to first reduce ourselves to the situation of Theorem \ref{main1}, where the graphs are contained in $\mathbb{R} \times \mathbb{R}_-$. This can be achieved by constructing a family of increasing functions $\{\widehat{k}^n(t,\cdot)\}_{n \geq 1}$ with the same domain as $\{k^n(t,\cdot)\}_{n \geq 1}$, defined by $\widehat{k}^n(t,x) := k^n(t,x) \wedge n - n$ for every point $x$. The associated operators $\{\widehat{k}^n_t(\cdot)\}_{n \geq 1}$ then have their graphs contained in $\mathbb{R} \times \mathbb{R}_-$. Thus, by Theorem \ref{main1}, for each $n \geq 1$, there exists a unique quadruple of processes $(\widehat{Y}^n, \widehat{Z}^n, \widehat{\psi}^n, \widehat{K}^n)$ that solves the MBSDE$(\xi, f - n, \widehat{k}^n)$ and satisfies
$$
\mathbb{E}\left[\sup_{0 \leq t \leq T} |\widehat{Y}^n_t|^2 + \int_{0}^{T} \left( |\widehat{Z}^n_s|^2 + \|\widehat{\psi}^n_s\|^2_\pi \right) ds + |\widehat{K}^n_T|^2 \right] < +\infty.
$$

If we set $(Y^n_t, Z^n_t, \psi^n_t, K^n_t):=(\widehat{Y}^n_t, \widehat{Z}^n_t, \widehat{\psi}^n_t, \widehat{K}^n_t - nt)$ for every $t \in [0,T]$, then we obtain that $Y^n$ satisfies
$$
Y^n_t = \xi + \int_t^T f(s, Y^n_s, Z^n_s, \psi^n_s) ds - \int_t^T Z^n_s\, dW_s - \int_t^T\int_{\mathcal{U}} \psi^n_s(e)\, \tilde{N}(ds,de)+(K^n_T - K^n_t).
$$ 
Thus, \eqref{eq1} is clearly satisfied, and we can also see that $(Y^n, K^n)$ satisfies \eqref{eq2}--\eqref{eq3}. 

Now, let us consider an arbitrary pair of optional processes $(\alpha_t,\beta_t)$ that belongs to $\operatorname{Gr}(k^n_t)$. From the definition of $k^n_t(\cdot)$ and $\widehat{k}^n_t(\cdot)$, we clearly have $(\alpha_t,\widehat{\beta}^n_t):=(\alpha_t,\beta_t - nt) \in \operatorname{Gr}(\widehat{k}^n_t)$. Therefore,
$$
(Y^n_t - \alpha_t)(dK^n_t + \beta_t dt) = (\widehat{Y}^n_t - \alpha_t)(d\widehat{K}^n_t + \widehat{\beta}^n_t dt),
$$
which is a.s. negative on $[0,T]$ from \eqref{eq4}. Due to the uniqueness result (Theorem \ref{uniq}), we deduce that $(Y^n, Z^n, \psi^n, K^n)$ is the unique solution to the MBSDE$(\xi, f, {k}^n)$.

Finally, using assumption (C), we can piece together the various processes $(Y^n, Z^n, \psi^n, K^n)$, each defined locally, to construct a global process $(Y, Z, \psi, K)$ that satisfies the MBSDE$(\xi, f, k)$.

 \begin{proof}[Proof of Theorem \ref{main2}]
 We consider the sequence of operators $\{k^n_t(\cdot)\}_{n \geq 1}$ with the same domain $\mathcal{D}_t$ as $k_t(\cdot)$, associated with the increasing function defined by $k_n(t,\cdot):=k(t,\cdot) \wedge n$, as discussed above, and let $(Y^n, Z^n, \psi^n, K^n)$ denote the solution of the MBSDE$(\xi,f,{k}^n)$ for each $n \geq 1$. Since $k_{n+1}(t,x) \geq k_n(t,x)$ for all $x \in ]a_t; +\infty[$, it follows from the comparison Theorem \ref{thm1} that $Y^{n+1}_t \leq Y^n_t$ for all $t \in [0,T]$ a.s. 
 
 Following the proof of Theorem 4 in \cite{marois1990equations}, we can consider, for each fixed $t \in [0,T]$, a real-valued sequence $\{x_{n,t}\}_{n \geq 1}$ defined by:
 \begin{equation*}
 	x_{n,t} =
 	\left\lbrace 
 	\begin{split}
 		&\inf\left\{x \in \mathbb{R} : \ell(t,x) \geq n\right\},\\
 		&+\infty, \quad\qquad\text{ if } ~ \left\{x \in \mathbb{R} : \ell(t,x) \geq n\right\}=\emptyset.
 	\end{split}
 	\right. 
 \end{equation*}
 This sequence is increasing and satisfies $\lim\limits_{n \rightarrow +\infty} x_{n,t} = +\infty$, since $\ell(t,\cdot)$ takes finite values. Additionally, we define a sequence of stopping times $\{\tau_n\}_{n \in \mathbb{N}}$ as follows: for each $n \geq 1$,
 \begin{equation*}
 	\tau_{n} =
 	\left\lbrace 
 	\begin{split}
 		&\inf\left\{t \in [0,T] : \ell(t,Y^n_t) \leq n\right\},\\
 		&+\infty, \quad\qquad\text{ if } ~ \left\{t \in [0,T] : \ell(t,Y^n_t) \leq n\right\}=\emptyset.
 	\end{split}
 	\right. 
 \end{equation*}
 Since $Y^{n+1} \leq Y^n$ and $\ell$ is increasing in its second variable, we have $\ell(t, Y^{n+1}_t) \leq \ell(t, Y^n_t)$, which implies $\tau_{n+1} \leq \tau_n$. Hence, the sequence $\{\tau_n\}_{n \in \mathbb{N}}$ is decreasing and satisfies $\tau_0 = T$ by assumption (C). 
 
 Moreover, for any $t \in [0,T]$, using Markov's inequality, we have:
 \begin{equation*}
 	\begin{split}
 		\mathbb{P}\left(\tau_{n} > t\right) \leq \mathbb{P}\left(\ell(t,Y^n_t) > n\right)
 		&\leq \mathbb{P}\left(\ell(t,Y^1_t) > n\right)\\
 		&\leq \frac{1}{n^2} \mathbb{E}\left[\sup_{0 \leq t \leq T} |\ell(t,Y^1_t)|^2\right]\\
 		&\leq \frac{2C}{n^2} \left(1 + \mathbb{E}\left[\sup_{0 \leq t \leq T} |Y^1_t|^2\right]\right).
 	\end{split}
 \end{equation*}
 Consequently, since the quantity $\left(1 + \mathbb{E}\left[\sup_{0 \leq t \leq T} |Y^1_t|^2\right]\right)$ is finite, we conclude that $\tau_n \searrow 0$ a.s. as $n \rightarrow +\infty$.
 
Let $n > 1$. Remark that for $s \in ]\tau_{n-1}, T]$, we have $\ell(t,Y^{n-1}_t) \leq n-1$ a.s. Then, from the definition of the sequence $x_{n-1,s}$, we have $Y^{n-1}_s \leq x_{n-1,s}$ and in particular, we also get  $Y^{n}_s \leq x_{n-1,s}$ as $Y^{n}_s \leq Y^{n-1}_s$. Therefore, we get $\frac{Y^{n-1}_s+Y^{n}_s}{2} \leq x_{n-1,s}$, yielding that  $\ell\big(s,\frac{Y^{n-1}_s+Y^{n}_s}{2} \big) < n-1$. Consequently, from assumption (C), we have $k_{n-1}\big(s,\frac{Y^{n-1}_s+Y^{n}_s}{2} \big)=k\big(s,\frac{Y^{n-1}_s+Y^{n}_s}{2} \big)$. Similarly, as the sequence $\{x_{n,s}\}_{n \geq 1}$ is increasing we have $Y^{n}_s \leq Y^{n-1}_s \leq x_{n,s}$, thus $\ell\big(s,\frac{Y^{n-1}_s+Y^{n}_s}{2} \big) < n$ and then $k_{n}\big(s,\frac{Y^{n-1}_s+Y^{n}_s}{2} \big)=k\big(s,\frac{Y^{n-1}_s+Y^{n}_s}{2} \big)$. The quadruplet $(Y^n, Z^n, \psi^n, K^n)$ and $(Y^{n-1}, Z^{n-1}, \psi^{n-1}, K^{n-1})$ are respective solutions of MBSDE$(\xi,f, k_n)$ and MBSDE$(\xi,f, k_{n-1})$, respectively. By following the above argumentation, and performing the same arguments as the one used in the proof of Lemma \ref{lemma1}, we can derive that the measure $\mathds{1}_{\{\tau_{n-1}< s\leq T\}}(Y^{n-1}_s-Y^{n}_s)(dK^{n-1}_s-dK^{n}_s)$ is a.s. negative. Additionally, by the continuity of $K^n$ and $K^{n+1}$, we can deduce that $\mathds{1}_{\{\tau_{n-1} \leq s \leq  T\}}(Y^{n-1}_s-Y^{n}_s)(dK^{n-1}_s-dK^{n}_s) \leq 0$ a.s.\\
By applying It\^o's formula to $(t, x) \mapsto e^{\zeta t } (Y^{n-1}_t-Y^{n}_t)^2$ on the time interval $[\tau_{n-1}, T]$ and re-performing similar calculations as in Theorem \ref{uniq} on the other hand, we can derive that $(Y^n_t, Z^n_t, \psi^n_t, K^n_t)=(Y^{n-1}_t, Z^{n-1}_t, \psi^{n-1}_t, K^{n-1}_t)$ for all $t \in [\tau_{n-1},T]$.\\
Next let us introduce the following processes:
\begin{equation*}
	\begin{split}
		Y_t &=\sum_{n=1}^{+\infty} \mathds{1}_{[\tau_{n}, \tau_{n-1}]}(t) Y^n_t,\quad Z_t=Z^1_t \mathds{1}_{]\tau_{1},T]}+\sum_{n=2}^{+\infty} \mathds{1}_{]\tau_{n}, \tau_{n-1}]}(t) Z^{n}_t\\
		\psi_t&=\psi^1_t \mathds{1}_{]\tau_1,T]}(t)+\sum_{n=2}^{+\infty} \mathds{1}_{]\tau_{n}, \tau_{n-1}]}(t) \psi^{n}_t,\quad K_t=\sum_{n=1}^{+\infty} \mathds{1}_{[\tau_{n}, \tau_{n-1}]}(t) K^n_t,\qquad t \in [0,T].
	\end{split}
\end{equation*}
Since the values at $t=0$ do not affect the stochastic integrals $\int_0^\cdot Z_s dW_s$ or $\int_0^\cdot \int_\mathcal{U} \psi_s(e) \tilde{N}(ds,de)$, we may, without loss of generality, set $Z_0=0$ and $\psi_0=0$.

The construction of $(Y, Z, \psi, K)$ consists simply of the concatenation of $\{Y^n, Z^n, \psi^n, K^n\}$, which, by its very definition, clearly satisfies the MBSDE$(\xi,f,k)$.
 \end{proof}
\begin{remark}
Compared to the definition of the solution given in \ref{def}, we can observe that the constructed solution in the proof of Theorem \ref{main2}, under the additional assumption (C), satisfies
$$
\int_{0}^{T} \left( |Z_s|^2 + \|\psi_s\|^2_\pi \right) ds + K_T < +\infty, \quad \mathbb{P}\text{-a.s.}
$$
where $K$ is continuous and has bounded variation on $[0,T]$ with $K_0=0$. Under this integrability, the uniqueness result stated in Theorem \ref{uniq} remains valid, even in the presence of integrability issues. Indeed, the same Itô's formula can still be applied, together with Lemma \ref{lemma1}, and by employing a localization argument for the local martingale part, followed by the use of Fatou's and Gronwall's lemmas, we obtain the uniqueness of the solution.

The approach adopted in this section brings to mind the notion of local solutions for classical BSDEs (without the associated maximal monotone operators $k_t(\cdot)$) with double reflection, as extensively studied in the literature. In those works, the concatenation of solutions defined on each subinterval of $[0,T]$ is used to construct a global solution that exhibits the same lack of integrability, with square-integrability being ensured only locally (see, e.g., \cite{hamadene2010backward,HassaniSaid,hamadene2009bsdes}).
\end{remark}

\end{document}